\documentclass[11pt]{article}

\usepackage[margin=1in]{geometry}
\usepackage{amsmath,amssymb,amsthm}
\RequirePackage[colorlinks,citecolor=blue,urlcolor=blue]{hyperref}
\usepackage{booktabs}
\usepackage{array}
\usepackage{microtype}
\usepackage{pgfplots}
\pgfplotsset{compat=1.18}

\title{Dunnett’s and Tukey’s procedures under mixed model covariance structures: dimensional reduction of multivariate‑t integrals}
\author{Dmytro Ivanchenko, Rostyslav Yamnenko \\ \small Taras Shevchenko National University of Kyiv, Ukraine}
\date{19.10.2026}

\newtheorem{proposition}{Proposition}

\newcommand{\E}{\operatorname{E}}
\newcommand{\Var}{\operatorname{Var}}
\newcommand{\Cov}{\operatorname{Cov}}
\newcommand{\Corr}{\operatorname{Corr}}
\newcommand{\Prob}{\operatorname{P}}
\renewcommand{\d}{\textrm{d}}

\begin{document}
\maketitle

\begin{abstract}
Dunnett's test compares several treatments with one control, and Tukey's test
compares every pair for independent samples with a common variance.
We investigate which designs and
covariances for a linear mixed model remain on the ordinary Dunnett and Tukey integrals, and in which
direction the constants move when the correlation changes.
It is shown that for Dunnett's procedure AR(1), the Toeplitz band, alternating ante-dependence, and the
random slope have roots below the ordinary analysis-of-variance value, while heterogeneous compound symmetry
and heteroscedastic groups have roots above that value.  For Tukey's procedure every computed
non-spherical covariance lowers the root.

The second goal is the dimensional reduction of these multivariate-\(t\)
integrals: a normal probability mixed over one chi-square scale.
Conditioning on the latent variables shared by the contrasts replaces that
probability by a product of univariate normals, and the integral has the
dimension of the conditioning vector plus one. Balanced compound symmetry
leaves a double integral. The studentized range reduces to the
Pearson--Hartley integral. The reduction stops for asymmetric overlap of
four or more treatments, for four or more autoregressive or unstructured
follow-up times, for pair-dependent variances or correlated disjoint pairs,
and for a separate variance estimate in each group.
\end{abstract}

\section{Introduction}
\label{sec:intro}

Dunnett's procedure~\cite{Dunnett1955} compares \(p\) treatments with one
control at a joint confidence level \(P\). Its critical values are the roots
of the one-sided integral~\eqref{eq:6} and the two-sided
integral~\eqref{eq:7}, derived below for independent, normally distributed samples with a common
variance. Equal sample sizes produce correlation \(1/2\).
Dunnett~\cite{Dunnett1964} gives the exact two-sided tables at that
correlation, and notes that when blocks are random the error variance is
the interaction mean square. In a balanced mixed model the subject effects
cancel in the comparisons with the control. For Dunnett's contrasts that
cancellation leaves the correlation equal to \(1/2\) for every intra-class
correlation. That identity is derived below. Hochberg and
Tamhane~\cite{HochbergTamhane1983} prove the Tukey counterpart in
Scheff{\'e}'s symmetric mixed model for a balanced two-way layout: both
intra-class correlations drop out of the pairwise comparisons among the
levels of the fixed factor, and the studentized range that uses the
interaction mean square is exact. 
The first goal of this paper is to determine which other designs and covariance structures remain on that same integral, and in which direction the constants in \cite{Dunnett1955, Dunnett1964}, or the root at correlation \(1/2\), deviate when the correlation structure
changes. That root is also the critical value of the ordinary analysis of
variance. Balanced compound symmetry leaves the correlation at \(1/2\), so
at a given \(\nu\) the critical constant is unchanged.

The same question applies to Tukey's honestly significant
difference~\cite{Tukey1953},
which compares every pair of treatments. Its critical value is a percentage point
of the studentized range, derived below. Any spherical covariance, balanced
compound symmetry included, leaves the pairwise correlation unchanged, so at
a given \(\nu\) the critical constant is the ordinary Tukey constant. A
covariance that moves those correlations changes the root, and a maximum
absolute pairwise statistic lying between the two roots is a family on which
the mixed-model procedure and the ordinary Tukey procedure disagree.

The second goal is to reduce multivariate $t$ integrals to
a form that can be computed effectively. Dunnett's integrals~\eqref{eq:6} and~\eqref{eq:7} are \(p\)-dimensional normal
probabilities mixed over an independent chi-square scale. Whenever the
contrasts share a small number of latent variables, conditioning on those
variables turns the \(p\)-dimensional probability into a product of
univariate normal probabilities, and the integral has the dimension of the
conditioning vector plus one. In the balanced compound-symmetry model that
dimension is two. The same reduction is carried through unbalanced and
incomplete layouts, \(N\)-way repeated measures on one or more factors, and the other
covariance structures of mixed-model practice, and the paper studies where
the reduction stops. The studentized-range probability is reduced below to
the Pearson-Hartley integral. Section~\ref{sec:oneway} shows that sphericity
leaves the corresponding critical constant equal to the ordinary Tukey
constant, and records the covariances that move it.
Section~\ref{sec:nway} writes both probabilities for an \(N\)-way layout
with \(M\) within-subjects factors.

\subsection{Dunnett's integrals}
\label{sec:integral}

Let treatment \(0\) be the control and let treatments \(1,\ldots,p\) be the
experimental treatments. Write \(\bar X_i\) for the mean of \(n_i\) independent
normal observations on treatment \(i\), each with variance \(\sigma^2\), and
let \(s^2\) be an estimator of \(\sigma^2\), independent of the means and based
on \(\nu\) degrees of freedom. Dunnett~\cite[Section~III]{Dunnett1955} forms
the studentized comparisons
\begin{equation}
\label{eq:z}
z_i
=
\frac{(\bar X_i-\bar X_0)-(\mu_i-\mu_0)}
{\sqrt{1/n_i+1/n_0}},
\qquad
t_i=\frac{z_i}{s},
\qquad i=1,\ldots,p.
\end{equation}
The denominator is the standard deviation of \(\bar X_i-\bar X_0\), divided
by \(\sigma\). It removes the sample sizes from the variance, so that each
\(z_i\) has variance \(\sigma^2\) rather than
\(\sigma^2(1/n_i+1/n_0)\). Dividing by \(s\) then produces a multivariate
analogue of Student's \(t\), in the sense of Dunnett and
Sobel~\cite{DunnettSobel1954} and Cornish~\cite{Cornish1954}: a normal vector
studentized by an independent root-mean-square estimator of its scale.

Because the treatment samples are independent, the only covariance among the
differences \(\bar X_i-\bar X_0\) comes from the shared control mean,
\(\Cov(\bar X_i-\bar X_0,\bar X_j-\bar X_0)=\sigma^2/n_0\). Dividing by the
two standard deviations gives Dunnett's correlation
\begin{equation}
\label{eq:rho-oneway}
\rho_{ij}
=
\frac{1/n_0}
{\sqrt{(1/n_i+1/n_0)(1/n_j+1/n_0)}}
=
\bigl[(n_0/n_i+1)(n_0/n_j+1)\bigr]^{-1/2}.
\end{equation}
The first expression is the definition of correlation; the second is the form
printed by Dunnett~\cite[p.~1103]{Dunnett1955}. Equal sample sizes make every
factor \(n_0/n_i+1\) equal to \(2\), so \(\rho_{ij}=1/2\).

One-sided limits with joint confidence coefficient \(P\) use constants
\(d_i'\) satisfying Dunnett~\cite[equation~(4)]{Dunnett1955},
\begin{equation}
\label{eq:4}
\Prob(t_1<d_1',\ldots,t_p<d_p')=P.
\end{equation}
A lower limit for \(\mu_i-\mu_0\) is
\(\bar X_i-\bar X_0-d_i' s\sqrt{1/n_i+1/n_0}\). The event that every such
limit lies below the corresponding parameter is exactly the event in
\eqref{eq:4}. An upper limit uses the same constant on the other side, by
symmetry of the centered normal vector. Conditioning on the value of \(s\)
turns~\eqref{eq:4} into
\begin{equation}
\label{eq:6}
P
=
\int_{-\infty}^{\infty}
F(d_1' s,\ldots,d_p' s)\,p(s)\,\d s.
\end{equation}
Here \(F(a_1,\ldots,a_p)=\Prob(z_1<a_1,\ldots,z_p<a_p)\) is the distribution
function of the normal vector \((z_1,\ldots,z_p)\), and \(p(s)\) is the density
of \(s\). The integral is the law of total probability: given \(s\), the event
\(\{t_i<d_i'\}\) is the event \(\{z_i<d_i' s\}\). Dunnett writes the range as
\((-\infty,\infty)\). Since \(s=\sigma\sqrt{\chi^2_\nu/\nu}\) is supported on
\((0,\infty)\), the integrand vanishes for \(s<0\).

For \(\rho_{ij}=1/2\) and equal coordinates, Dunnett evaluated \(F\) by
numerical integration to produce Tables~1a and~1b, which therefore solve~\eqref{eq:4}
and~\eqref{eq:6}. The infinite-degree row was taken from
Bechhofer~\cite{Bechhofer1954}. Intermediate degrees of freedom were
interpolated in \(1/\nu\).

Two-sided limits use constants \(d_i''\) satisfying
Dunnett~\cite[equation~(5)]{Dunnett1955},
\begin{equation}
\label{eq:5}
\Prob\bigl(|t_1|<d_1'',\ldots,|t_p|<d_p''\bigr)=P,
\end{equation}
and the interval is
\(\bar X_i-\bar X_0\pm d_i'' s\sqrt{1/n_i+1/n_0}\).
The same conditioning produces
\begin{equation}
\label{eq:7}
P
=
\int_{-\infty}^{\infty}
G(d_1'' s,\ldots,d_p'' s)\,p(s)\,\d s,
\end{equation}
where \(G(a_1,\ldots,a_p)=\Prob(|z_1|<a_1,\ldots,|z_p|<a_p)\) is the
distribution function of the absolute values \(|z_i|\). The function \(G\)
replaces \(F\) because both tails of each coordinate are constrained. No table
of \(G\) was available in 1955, so Tables~2a and~2b do not solve~\eqref{eq:7}.

Throughout this note the constants are taken equal,
\(d_1'=\cdots=d_p'=d'\) and \(d_1''=\cdots=d_p''=d''\). Dunnett adopts this
choice because it simplifies the tables and, when \(\rho_{ij}=\rho>0\),
minimizes \(\sum d_i'^2\) and \(\sum d_i''^2\) among solutions of~\eqref{eq:4}
and~\eqref{eq:5}~\cite[p.~1105]{Dunnett1955}. That sum-of-squares argument
uses one common correlation \(\rho>0\). Once the loadings differ, a single
constant is kept as a convention, so that each family has one critical value.

Until \(G\) could be computed, Dunnett and Sobel~\cite{DunnettSobel1955} gave
product bounds. The one-sided bound,
Dunnett~\cite[equation~(8)]{Dunnett1955}, replaces the joint
probability by the product of its marginal \(t\) probabilities. The sharper
one-sided bound, Dunnett~\cite[equation~(10)]{Dunnett1955}, uses the
bivariate probability:
\begin{equation}
\label{eq:10}
\Prob(t_i<d'\ \text{for all}\ i\mid\rho_{ij}=\rho)
\ge
\bigl[
\Prob(t_1<d',\,t_2<d'\mid\rho_{12}=\rho)
\bigr]^{p/2}.
\end{equation}
The two-sided companion is Dunnett~\cite[equation~(11)]{Dunnett1955},
\begin{equation}
\label{eq:11}
\Prob\bigl(|t_i|<d''\ \text{for all}\ i\bigm|\rho_{ij}=\rho\bigr)
\ge
\bigl[
\Prob\bigl(|t_1|<d'',\,|t_2|<d''\bigm|\rho_{12}=\rho\bigr)
\bigr]^{p/2}.
\end{equation}
Both inequalities become equalities at \(p=1\) and at \(p=2\), where the
joint event is already the bivariate event and the exponent is \(1\). For
\(p\ge 3\), a critical value that makes the right-hand side equal to \(P\)
produces a joint probability strictly larger than \(P\). The footnotes to
Tables~2a and~2b state that those tables solve the right-hand side
of~\eqref{eq:11}. There is no corresponding gap for Tables~1a and~1b, because
those tables were computed from~\eqref{eq:6}.

\subsection{Reduction of the integral}
\label{sec:reduction}

Let \(U,V_1,\ldots,V_p\) be independent standard normal variables. Dunnett's
Algorithm AS~251~\cite[equation~(1), p.~565]{Dunnett1989} writes the
correlation as a product,
\begin{equation}
\label{eq:product}
\rho_{ij}=b_i b_j,
\qquad i\neq j,\quad -1<b_i<1,
\end{equation}
and represents the standardized coordinates by
\begin{equation}
\label{eq:product-rep}
\frac{z_i}{\sigma}
=
b_i\,U+\sqrt{1-b_i^2}\,V_i.
\end{equation}
Independence of \(U\) and the \(V_i\) gives
\(\Cov(z_i/\sigma,z_k/\sigma)=b_i b_k\) for \(i\neq k\), and each coordinate
has variance \(b_i^2+(1-b_i^2)=1\).
When every correlation equals a common value \(\rho\in[0,1)\), the same page
sets each loading equal to \(\sqrt{\rho}\). Substitution
in~\eqref{eq:product-rep} produces the equicorrelated representation
\begin{equation}
\label{eq:factor}
\frac{z_i}{\sigma}
=
\sqrt{\rho}\,U+\sqrt{1-\rho}\,V_i.
\end{equation}
The coordinates are standard normal, and \(\Cov(z_i/\sigma,z_k/\sigma)=\rho\)
for \(i\neq k\).
For \(\rho=1/2\) the common loading is \(1/\sqrt{2}\), and the representation
is that of Gupta and Sobel~\cite{GuptaSobel1957}, with \(U\) the standardized
control mean.

Put \(w=s/\sigma\). Conditionally on \(U=u\) and on \(w\), independence of
the \(V_i\) factors the one-sided event in~\eqref{eq:6} as
\begin{equation}
\label{eq:cond-one}
\Prob(z_i<d's\text{ for all }i\mid U=u,\,w)
=
\left[
\Phi\!\left(\frac{d'w-\sqrt{\rho}\,u}{\sqrt{1-\rho}}\right)
\right]^p
\end{equation}
and the two-sided event in~\eqref{eq:7} as
\begin{equation}
\label{eq:cond-two}
\left[
\Phi\!\left(\frac{d''w+\sqrt{\rho}\,u}{\sqrt{1-\rho}}\right)
-
\Phi\!\left(\frac{-d''w+\sqrt{\rho}\,u}{\sqrt{1-\rho}}\right)
\right]^p.
\end{equation}
Here \(\Phi\) is the standard normal distribution function. The form
\eqref{eq:cond-two} replaces \(u\) by \(-u\), which leaves the integral
against the symmetric density of \(U\) unchanged. Removing the conditioning
produces the double integral
\begin{equation}
\label{eq:reduced}
\begin{aligned}
P(d'')
&=
\int_0^{\infty}\int_{-\infty}^{\infty}
\left[
\Phi\!\left(\frac{d''w+\sqrt{\rho}\,u}{\sqrt{1-\rho}}\right)
-
\Phi\!\left(\frac{-d''w+\sqrt{\rho}\,u}{\sqrt{1-\rho}}\right)
\right]^p
\phi(u)\,f_\nu(w)\,\d u\,\d w,
\end{aligned}
\end{equation}
and the one-sided integral
\begin{equation}
\label{eq:reduced-one}
P(d')
=
\int_0^{\infty}\int_{-\infty}^{\infty}
\left[
\Phi\!\left(\frac{d'w-\sqrt{\rho}\,u}{\sqrt{1-\rho}}\right)
\right]^p
\phi(u)\,f_\nu(w)\,\d u\,\d w.
\end{equation}
At \(\rho=1/2\) the two-sided integral is formula~(7.2) of Gupta and
Sobel~\cite[formula~(7.2)]{GuptaSobel1957}, and the one-sided integral is
their formula~(4.2)~\cite[formula~(4.2)]{GuptaSobel1957}.
Here \(\phi\) is the standard normal density and \(f_\nu\) is the density of
\(w=\sqrt{\chi^2_\nu/\nu}\),
\begin{equation}
\label{eq:wdensity}
f_\nu(w)
=
\frac{\nu^{\nu/2}}{2^{\nu/2-1}\,\Gamma(\nu/2)}
\,
w^{\nu-1}
\exp\!\bigl(-\nu w^2/2\bigr),
\qquad w>0.
\end{equation}
As \(\nu\to\infty\), \(w\) concentrates at \(1\) and each
double integral reduces to a single normal integral.

\subsection{Pairwise contrasts and the studentized range}
\label{sec:tukey-range}

Let \(\bar Y_1,\ldots,\bar Y_k\) be the means of \(k\) independent normal
samples of size \(n\), with common variance \(\sigma^2\), and let \(s^2\) be
an independent estimator of \(\sigma^2\) on \(\nu\) degrees of freedom. The
studentized range is
\begin{equation}
\label{eq:range}
Q
=
\frac{\max_i\bar Y_i-\min_i\bar Y_i}{s/\sqrt{n}}.
\end{equation}
The Tukey critical value \(q_{P,k,\nu}\) solves \(\Prob(Q\le q_{P,k,\nu})=P\).
An equivalent description uses the \(m=k(k-1)/2\) pairwise statistics
\begin{equation}
\label{eq:tij}
z_{ij}
=
\frac{\bar Y_i-\bar Y_j}{\sqrt{2/n}},
\qquad
t_{ij}=\frac{z_{ij}}{s},
\qquad
1\le i<j\le k.
\end{equation}
Dividing by \(\sqrt{2/n}\) removes the sample size, so each \(z_{ij}\) has
variance \(\sigma^2\). The largest absolute pairwise difference is the range,
so
\begin{equation}
\label{eq:d-from-q}
\max_{i<j}|t_{ij}|
=
\frac{Q}{\sqrt{2}}.
\end{equation}
The critical constant on the scale of Dunnett's \(d''\) is therefore
\(d_T=q_{P,k,\nu}/\sqrt{2}\). Throughout, \(P=0.95\), and \(d_T\) denotes
that root.

Because the samples are independent and of equal size, the correlation of
\((z_{ij})_{i<j}\) is
\begin{equation}
\label{eq:tukey-corr}
\Corr(\bar Y_i-\bar Y_j,\bar Y_a-\bar Y_b)
=
\begin{cases}
\pm 1/2, & \text{the pairs share a treatment},\\
0, & \text{the pairs are disjoint}.
\end{cases}
\end{equation}
The sign is fixed by the orientation \(\bar Y_i-\bar Y_j\): two contrasts that
both leave the shared treatment, or both enter it, have correlation \(+1/2\),
and two contrasts that pass through the shared treatment in opposite directions
have correlation \(-1/2\). For \(k=3\) every two pairs share a treatment, so
every absolute correlation is \(1/2\). For \(k\ge 4\) the disjoint pairs
contribute exact zeros. The \(m\) contrasts span only \(k-1\) dimensions, since
they sum to zero around every closed loop, and the correlation matrix is
singular for every \(k\ge 3\).

Conditioning on \(s\), as in~\eqref{eq:6} and~\eqref{eq:7}, gives the
coverage integrals. One-sided limits with a common constant \(d_T'\) satisfy
\begin{equation}
\label{eq:tukey-one}
P
=
\int_{-\infty}^{\infty}
F(d_T' s,\ldots,d_T' s)\,p(s)\,\d s,
\end{equation}
where \(F=\Prob(z_{ij}<a_{ij}\text{ for all }i<j)\) has one argument for each
pair. Two-sided limits with the constant \(d_T\) satisfy
\begin{equation}
\label{eq:tukey-two}
P
=
\int_{-\infty}^{\infty}
G(d_T s,\ldots,d_T s)\,p(s)\,\d s,
\end{equation}
where \(G=\Prob(|z_{ij}|<a_{ij}\text{ for all }i<j)\).
By~\eqref{eq:d-from-q}, the event in~\eqref{eq:tukey-two} is
\(\{Q\le q_{P,k,\nu}\}\).

\subsection{Reduction of the range probability}
\label{sec:tukey-reduce}

Equation~\eqref{eq:tukey-two} is \(\Prob(Q\le q)\) at \(q=d_T\sqrt{2}\).
Let \(Z_1,\ldots,Z_k\) be independent standard normals, \(W=\max_i Z_i-\min_i Z_i\). Integrating their joint density over
\(W\le q\) produces the probability integral tabulated by Pearson and
Hartley~\cite{PearsonHartley1943},
\begin{equation}
\label{eq:pearson-hartley}
\Prob(W\le q)
=
k\int_{-\infty}^{\infty}
\phi(x)\,
\bigl[\Phi(x+q)-\Phi(x)\bigr]^{k-1}
\,\d x.
\end{equation}
The factor \(k\) counts which sample mean falls at the lower end of the
interval of length \(q\). The integrand is then the probability that the other
\(k-1\) means fall inside that interval. Mixing over \(w=s/\sigma\), with
density~\eqref{eq:wdensity}, replaces \(q\) by \(qw\):
\begin{equation}
\label{eq:tukey-student}
\begin{aligned}
\Prob(Q\le q)
&=
\int_0^{\infty}
k\int_{-\infty}^{\infty}
\phi(x)\,
\bigl[\Phi(x+qw)-\Phi(x)\bigr]^{k-1}
\,\d x
\,
f_\nu(w)\,\d w.
\end{aligned}
\end{equation}
At \(q=d_T\sqrt{2}\) this is~\eqref{eq:tukey-two}. The inner integral is
one-dimensional, so the probability of the \(m\) pairs is a double integral
for every \(k\). As \(\nu\to\infty\), \(w\) concentrates at \(1\)
and~\eqref{eq:tukey-student} reduces to~\eqref{eq:pearson-hartley}.

\section{One-way linear mixed model}
\label{sec:oneway}

The one-way mixed model has a single fixed treatment factor crossed with a
random subject. Compound symmetry is the leading case: the contrast
correlation is \(1/2\), the interaction mean square studentizes the
comparisons, and integrals~\eqref{eq:6} and~\eqref{eq:7} reduce to double
integrals. The same skeleton covers unequal replication, incomplete subject
sets, and the other covariance structures of mixed-model practice.
Sphericity leaves the Tukey constant equal to the ordinary
studentized-range value.

\subsection{Mixed model with compound symmetry}
\label{sec:model}

Consider the balanced random-intercept model
\begin{equation}
\label{eq:lmm}
Y_{ij}=\mu_i+b_j+\varepsilon_{ij},
\qquad
i=0,\ldots,p,
\quad
j=1,\ldots,n,
\end{equation}
with \(b_j\sim N(0,\sigma_b^2)\), \(\varepsilon_{ij}\sim N(0,\sigma^2)\), and all
random terms independent. This is the balanced mixed model used when every
subject, or every block, is observed once under the control and under each
treatment~\cite[Section~2]{Dunnett1964}. Marginally, the \(p+1\) observations
on subject \(j\) have compound symmetry, because a shared \(b_j\) contributes
the same covariance to every pair and the residuals contribute variance only
on the diagonal~\cite{Wolfinger1996}:
\[
\Var(Y_{ij})=\sigma_b^2+\sigma^2,
\qquad
\Corr(Y_{ij},Y_{i'j})
=
\rho_Y
=
\frac{\sigma_b^2}{\sigma_b^2+\sigma^2}
\quad(i\neq i').
\]
The comparisons with the control are
\[
D_i=\bar Y_{i\cdot}-\bar Y_{0\cdot},
\qquad
\E[D_i]=\mu_i-\mu_0.
\]
The subject effects appear in both means and cancel, so
\begin{equation}
\label{eq:contrast-moments}
\Var(D_i)=\frac{2\sigma^2}{n},
\qquad
\Cov(D_i,D_k)=\frac{\sigma^2}{n},
\qquad
\Corr(D_i,D_k)=\frac12.
\end{equation}
The correlation in integrals~\eqref{eq:6} and~\eqref{eq:7} is \(1/2\) for
every intra-class correlation \(\rho_Y\in[0,1)\). The value is determined by
the second moments.

\begin{proposition}
\label{prop:cs}
Let \(p\ge 1\) and let \(Y_0,\ldots,Y_p\) have common variance \(\tau^2>0\) and
common pairwise covariance \(\rho_Y\tau^2\), where \(-1/p\le\rho_Y<1\). Then, for all \(i,j\in\{0,\ldots,p\}\) with
\(i\neq j\),
\[
\Var(Y_i-Y_j)=2\tau^2(1-\rho_Y)>0,
\]
and, for all \(1\le i\neq k\le p\),
\[
\Cov(Y_i-Y_0,Y_k-Y_0)=\tau^2(1-\rho_Y),
\qquad
\Corr(Y_i-Y_0,Y_k-Y_0)=\frac12.
\]
For \(p\ge 2\) the contrasts with the control are therefore equicorrelated,
with correlation \(1/2\) for every such \(\tau^2\) and \(\rho_Y\).
\end{proposition}

\begin{proof}
Write \(\Sigma=\tau^2\bigl((1-\rho_Y)I_{p+1}+\rho_Y 11'\bigr)\) for the
covariance matrix of \((Y_0,\ldots,Y_p)'\). The eigenvalues of the matrix in
parentheses are \(1-\rho_Y\), of multiplicity \(p\), and \(1+p\rho_Y\), of
multiplicity one, so \(\Sigma\) is positive semidefinite if and only if
\(-1/p\le\rho_Y\le 1\). The restriction \(\rho_Y<1\) is the further condition
that a difference of two coordinates has positive variance. For \(i\neq j\),
\begin{align*}
\Var(Y_i-Y_j)
&=
\Var(Y_i)+\Var(Y_j)-2\Cov(Y_i,Y_j)
=
2\tau^2-2\rho_Y\tau^2
=
2\tau^2(1-\rho_Y).
\end{align*}
For \(1\le i\neq k\le p\), bilinearity gives
\begin{align*}
\Cov(Y_i-Y_0,Y_k-Y_0)
&=
\Cov(Y_i,Y_k)-\Cov(Y_i,Y_0)-\Cov(Y_k,Y_0)+\Var(Y_0)
\\
&=
\rho_Y\tau^2-\rho_Y\tau^2-\rho_Y\tau^2+\tau^2
=
\tau^2(1-\rho_Y).
\end{align*}
The two standard deviations equal \(\sqrt{2\tau^2(1-\rho_Y)}\), so
\[
\Corr(Y_i-Y_0,Y_k-Y_0)
=
\frac{\tau^2(1-\rho_Y)}{2\tau^2(1-\rho_Y)}
=
\frac12.
\]
\end{proof}

Sphericity is the second-moment condition that every pairwise difference has
the same positive variance. That condition forces the control contrasts to
have correlation \(1/2\).

\begin{proposition}
\label{prop:sphere}
Let \(p\ge 2\) and let \(Y_0,\ldots,Y_p\) have finite second moments. If
\(\Var(Y_i-Y_j)=v\) for some constant \(v>0\) and every \(0\le i<j\le p\),
then, for all \(1\le i\neq k\le p\),
\[
\Var(Y_i-Y_0)=v,
\qquad
\Cov(Y_i-Y_0,Y_k-Y_0)=\frac v2,
\qquad
\Corr(Y_i-Y_0,Y_k-Y_0)=\frac12.
\]
\end{proposition}

\begin{proof}
The variances \(\Var(Y_i-Y_0)\) and \(\Var(Y_k-Y_0)\) equal \(v\) by
hypothesis. For random variables \(A\) and \(B\) with finite second moments,
\(\Var(A-B)=\Var(A)+\Var(B)-2\Cov(A,B)\), and rearranging gives
\begin{equation}
\label{eq:polar}
2\Cov(A,B)=\Var(A)+\Var(B)-\Var(A-B).
\end{equation}
Set \(A=Y_i-Y_0\) and \(B=Y_k-Y_0\). Then \(A-B=Y_i-Y_k\), and the three
variances on the right of~\eqref{eq:polar} equal \(v\), so
\(2\Cov(A,B)=v\). Therefore \(\Cov(A,B)=v/2\) and
\(
\Corr(A,B)
=
\frac12.
\)
\end{proof}

Proposition~\ref{prop:cs} is the compound-symmetry case of
Proposition~\ref{prop:sphere}. The variance calculation in its proof shows
that \(\Var(Y_i-Y_j)=2\tau^2(1-\rho_Y)\) for every pair, so the hypothesis of
Proposition~\ref{prop:sphere} holds with \(v=2\tau^2(1-\rho_Y)\) and the
correlation is again \(1/2\). The treatment means in a balanced design have
this compound-symmetry covariance, and the same correlation follows. Under
sphericity, Huynh and Feldt~\cite{HuynhFeldt1970} show that the
treatment-by-subject interaction mean square is independent of the treatment
contrasts and is distributed as \(\sigma^2\chi^2_\nu/\nu\) on
\(\nu=p(n-1)\) degrees of freedom. The critical values at \(\rho=1/2\)
therefore apply to every balanced spherical covariance.

The residual mean square after fitting treatments and subjects is
\begin{equation}
\label{eq:mse}
s^2
=
\frac{1}{\nu}
\sum_{i=0}^{p}\sum_{j=1}^{n}
\bigl(Y_{ij}-\bar Y_{i\cdot}-\bar Y_{\cdot j}+\bar Y_{\cdot\cdot}\bigr)^2,
\qquad
\nu=p(n-1),
\end{equation}
and \(\nu s^2/\sigma^2\sim\chi^2_\nu\)~\cite[p.~28]{Scheffe1959}.
Dunnett~\cite{Dunnett1964} uses this interaction mean square as the error term
when the blocks are random. Therefore
\[
t_i
=
\frac{D_i-(\mu_i-\mu_0)}{s\sqrt{2/n}}
\]
has the \(p\)-variate \(t\) distribution of Section~\ref{sec:integral} with
\(\nu\) degrees of freedom and \(\rho_{ij}=1/2\), and integrals~\eqref{eq:6}
and~\eqref{eq:7} are the joint coverage probabilities.

The one-sided intervals, corresponding to~\eqref{eq:4}, are the lower limits
\begin{equation}
\label{eq:intervals-one}
\bar Y_{i\cdot}-\bar Y_{0\cdot}
-
d'\,s\sqrt{2/n},
\qquad i=1,\ldots,p,
\end{equation}
or the upper limits with the sign of \(d'\) changed. The two-sided intervals,
corresponding to~\eqref{eq:5}, are
\begin{equation}
\label{eq:intervals}
\bar Y_{i\cdot}-\bar Y_{0\cdot}
\pm
d''\,s\sqrt{2/n},
\qquad i=1,\ldots,p.
\end{equation}
The standard error \(s\sqrt{2/n}\) estimates the contrast scale
\(\sqrt{2\sigma^2/n}\).

Tables~\ref{tab:d95one} and~\ref{tab:d99one} give the one-sided roots \(d'\)
of~\eqref{eq:6} at \(\rho=1/2\), and Tables~\ref{tab:d95} and~\ref{tab:d99}
the two-sided roots \(d''\) of~\eqref{eq:7}. Rounded to two decimals, the
one-sided tables reproduce Tables~1a and~1b of
Dunnett~\cite{Dunnett1955}. The row \(\nu=\infty\) is the known-variance
case.

\begin{table}[ht]
\centering
\caption{Solutions \(d'\) of the one-sided integral~\eqref{eq:6} at
\(\rho_{ij}=1/2\) and joint confidence \(P=0.95\).}
\label{tab:d95one}
\footnotesize
\setlength{\tabcolsep}{4.5pt}
\begin{tabular}{crrrrrrrrr}
\toprule
\(\nu\backslash p\) & 1 & 2 & 3 & 4 & 5 & 6 & 7 & 8 & 9 \\
\midrule
5  & 2.015 & 2.440 & 2.681 & 2.848 & 2.976 & 3.078 & 3.163 & 3.236 & 3.300 \\
10 & 1.812 & 2.151 & 2.338 & 2.466 & 2.562 & 2.640 & 2.704 & 2.759 & 2.807 \\
20 & 1.725 & 2.027 & 2.192 & 2.304 & 2.389 & 2.456 & 2.512 & 2.559 & 2.601 \\
60 & 1.671 & 1.952 & 2.104 & 2.207 & 2.284 & 2.345 & 2.395 & 2.439 & 2.476 \\
\(\infty\) &
     1.645 & 1.916 & 2.062 & 2.160 & 2.234 & 2.292 & 2.340 & 2.381 & 2.417 \\
\bottomrule
\end{tabular}
\end{table}

\begin{table}[ht]
\centering
\caption{Solutions \(d'\) of the one-sided integral~\eqref{eq:6} at
\(\rho_{ij}=1/2\) and joint confidence \(P=0.99\).}
\label{tab:d99one}
\footnotesize
\setlength{\tabcolsep}{4.5pt}
\begin{tabular}{crrrrrrrrr}
\toprule
\(\nu\backslash p\) & 1 & 2 & 3 & 4 & 5 & 6 & 7 & 8 & 9 \\
\midrule
5  & 3.365 & 3.900 & 4.211 & 4.429 & 4.597 & 4.733 & 4.846 & 4.944 & 5.030 \\
10 & 2.764 & 3.115 & 3.314 & 3.453 & 3.559 & 3.644 & 3.715 & 3.777 & 3.830 \\
20 & 2.528 & 2.813 & 2.972 & 3.082 & 3.166 & 3.233 & 3.289 & 3.337 & 3.378 \\
60 & 2.390 & 2.638 & 2.775 & 2.869 & 2.940 & 2.997 & 3.044 & 3.085 & 3.120 \\
\(\infty\) &
     2.326 & 2.558 & 2.685 & 2.772 & 2.837 & 2.889 & 2.933 & 2.970 & 3.002 \\
\bottomrule
\end{tabular}
\end{table}

\begin{table}[ht]
\centering
\caption{Solutions \(d''\) of integral~\eqref{eq:7} at \(\rho_{ij}=1/2\) and
joint confidence \(P=0.95\).}
\label{tab:d95}
\footnotesize
\setlength{\tabcolsep}{4.5pt}
\begin{tabular}{crrrrrrrrr}
\toprule
\(\nu\backslash p\) & 1 & 2 & 3 & 4 & 5 & 6 & 7 & 8 & 9 \\
\midrule
5  & 2.571 & 3.030 & 3.293 & 3.476 & 3.615 & 3.727 & 3.821 & 3.900 & 3.970 \\
10 & 2.228 & 2.568 & 2.759 & 2.890 & 2.990 & 3.070 & 3.137 & 3.194 & 3.244 \\
20 & 2.086 & 2.379 & 2.540 & 2.651 & 2.735 & 2.802 & 2.857 & 2.905 & 2.946 \\
60 & 2.000 & 2.265 & 2.410 & 2.508 & 2.582 & 2.642 & 2.691 & 2.733 & 2.769 \\
\(\infty\) &
     1.960 & 2.212 & 2.349 & 2.442 & 2.511 & 2.567 & 2.613 & 2.652 & 2.686 \\
\bottomrule
\end{tabular}
\end{table}

\begin{table}[ht]
\centering
\caption{Solutions \(d''\) of integral~\eqref{eq:7} at \(\rho_{ij}=1/2\) and
joint confidence \(P=0.99\).}
\label{tab:d99}
\footnotesize
\setlength{\tabcolsep}{4.5pt}
\begin{tabular}{crrrrrrrrr}
\toprule
\(\nu\backslash p\) & 1 & 2 & 3 & 4 & 5 & 6 & 7 & 8 & 9 \\
\midrule
5  & 4.032 & 4.627 & 4.975 & 5.219 & 5.406 & 5.557 & 5.683 & 5.792 & 5.887 \\
10 & 3.169 & 3.531 & 3.739 & 3.883 & 3.994 & 4.084 & 4.159 & 4.223 & 4.279 \\
20 & 2.845 & 3.127 & 3.285 & 3.395 & 3.479 & 3.547 & 3.603 & 3.651 & 3.694 \\
\(\infty\) &
     2.576 & 2.794 & 2.915 & 2.998 & 3.060 & 3.110 & 3.152 & 3.188 & 3.219 \\
\bottomrule
\end{tabular}
\end{table}

Rounded to two decimals, Table~\ref{tab:d95} reproduces Table~II of
Dunnett~\cite{Dunnett1964} at \(\nu=5,10,20,60,\infty\). Table~\ref{tab:d99}
reproduces Table~III at \(\nu=10,20,\infty\). At \(\nu=5\) the same rounding
agrees with Table~III except at \(p=7\) and \(p=8\), where the table prints
\(5.69\) and \(5.80\) and the roots here are \(5.683\) and \(5.792\).

Table~\ref{tab:cover95} evaluates~\eqref{eq:7} at the constants of Dunnett's
Table~2a. For \(p\le 2\) the coverage is \(0.95\), up to rounding of the
published entries. For \(p\ge 3\) the constants exceed the roots
of~\eqref{eq:7}, and the joint coverage lies between about \(0.953\) and
\(0.969\).

\begin{table}[ht]
\centering
\caption{Joint coverage under integral~\eqref{eq:7}, at \(\rho_{ij}=1/2\), of
the two-sided constants in Dunnett's Table~2a.
}
\label{tab:cover95}
\footnotesize
\setlength{\tabcolsep}{4.5pt}
\begin{tabular}{crrrrrrrrr}
\toprule
\(\nu\backslash p\) & 1 & 2 & 3 & 4 & 5 & 6 & 7 & 8 & 9 \\
\midrule
5  & 0.9500 & 0.9500 & 0.9549 & 0.9586 & 0.9616 & 0.9638 & 0.9658 & 0.9674 & 0.9690 \\
10 & 0.9502 & 0.9501 & 0.9540 & 0.9561 & 0.9588 & 0.9600 & 0.9620 & 0.9633 & 0.9644 \\
20 & 0.9504 & 0.9501 & 0.9530 & 0.9548 & 0.9572 & 0.9584 & 0.9596 & 0.9607 & 0.9613 \\
\(\infty\) &
     0.9500 & 0.9497 & 0.9526 & 0.9536 & 0.9549 & 0.9567 & 0.9572 & 0.9573 & 0.9581 \\
\bottomrule
\end{tabular}
\end{table}

At \(\nu=5\) and \(p=9\), Table~2a gives \(4.49\) and the root
of~\eqref{eq:7} is \(3.970\), so the tabled interval is about \(13\%\) longer
and its coverage is \(0.969\). Figure~\ref{fig:table2a} plots the two sets of
constants. They meet at \(p=1\) and \(p=2\), where~\eqref{eq:11} is an
equality, and separate from \(p=3\).

\begin{figure}[ht]
\centering
\begin{tikzpicture}
\begin{axis}[
  width=0.48\textwidth, height=5.6cm,
  title={\(\nu=\infty\)},
  xlabel={\(p\)}, ylabel={\(d''\)},
  xmin=1, xmax=9, ymin=1.9, ymax=2.8,
  xtick={1,...,9},
  legend pos=south east, legend style={font=\footnotesize},
  grid=major, grid style={dotted},
  tick label style={font=\footnotesize}, label style={font=\small},
]
\addplot[thick, blue, mark=*, mark size=1.5pt] coordinates
  {(1,1.960) (2,2.212) (3,2.349) (4,2.442) (5,2.511) (6,2.567) (7,2.613) (8,2.652) (9,2.686)};
\addlegendentry{root of~\eqref{eq:7}}
\addplot[thick, red, dashed, mark=square*, mark size=1.5pt] coordinates
  {(1,1.96) (2,2.21) (3,2.37) (4,2.47) (5,2.55) (6,2.62) (7,2.67) (8,2.71) (9,2.75)};
\addlegendentry{Table~2a}
\end{axis}
\end{tikzpicture}
\hfill
\begin{tikzpicture}
\begin{axis}[
  width=0.48\textwidth, height=5.6cm,
  title={\(\nu=5\)},
  xlabel={\(p\)}, ylabel={\(d''\)},
  xmin=1, xmax=9, ymin=2.4, ymax=4.7,
  xtick={1,...,9},
  legend pos=south east, legend style={font=\footnotesize},
  grid=major, grid style={dotted},
  tick label style={font=\footnotesize}, label style={font=\small},
]
\addplot[thick, blue, mark=*, mark size=1.5pt] coordinates
  {(1,2.571) (2,3.030) (3,3.293) (4,3.476) (5,3.615) (6,3.727) (7,3.821) (8,3.900) (9,3.970)};
\addlegendentry{root of~\eqref{eq:7}}
\addplot[thick, red, dashed, mark=square*, mark size=1.5pt] coordinates
  {(1,2.57) (2,3.03) (3,3.39) (4,3.66) (5,3.88) (6,4.06) (7,4.22) (8,4.36) (9,4.49)};
\addlegendentry{Table~2a}
\end{axis}
\end{tikzpicture}
\caption{Two-sided constants at \(\rho_{ij}=1/2\) and \(P=0.95\): roots of
integral~\eqref{eq:7} from Table~\ref{tab:d95}, and the constants of Dunnett's
Table~2a, which solve the bound~\eqref{eq:11}. Left, \(\nu=\infty\); right,
\(\nu=5\).}
\label{fig:table2a}
\end{figure}
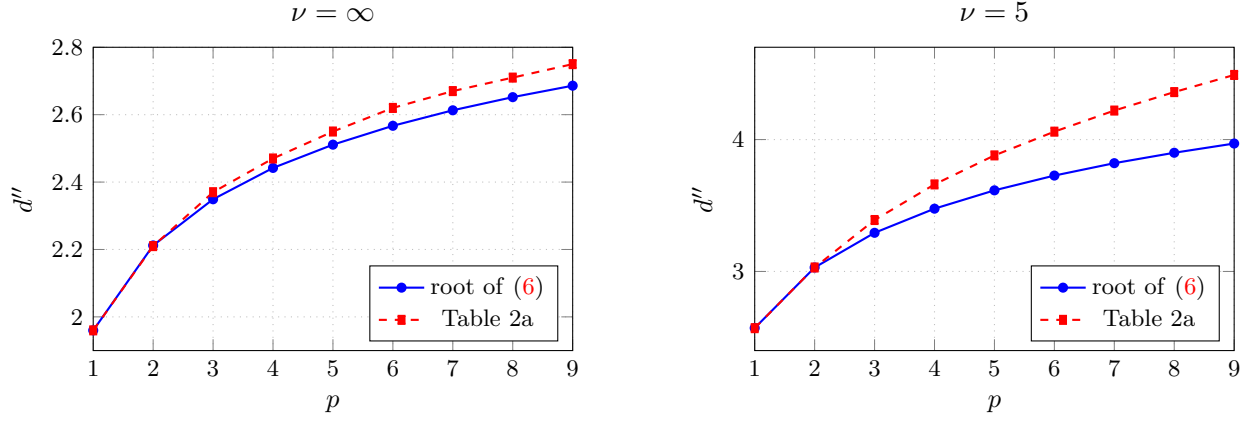

At \(99\%\) and \(\nu=\infty\), the Table~2b constants have coverages from
\(0.9901\) to \(0.9910\) for \(p=1,\ldots,9\), and the roots in
Table~\ref{tab:d99} differ from the printed constants by at most \(0.031\).
At \(\nu=5\) the gap reopens: the constants \(4.63\), \(5.09\), and \(6.53\)
at \(p=2\), \(3\), and \(9\) have coverages \(0.9900\), \(0.9909\), and
\(0.9937\), against the root \(5.887\) at \(p=9\).

Table~\ref{tab:rho} solves~\eqref{eq:reduced} at \(\nu=\infty\) and \(P=0.95\)
for several contrast correlations. The probability increases with \(\rho\),
so the critical value decreases.

\begin{table}[ht]
\centering
\caption{Solutions of~\eqref{eq:reduced} at infinite degrees of freedom and
\(P=0.95\).}
\label{tab:rho}
\begin{tabular}{crrrr}
\toprule
\(\rho\backslash p\) & 2 & 3 & 5 & 9 \\
\midrule
0    & 2.236 & 2.388 & 2.569 & 2.766 \\
0.25 & 2.231 & 2.379 & 2.556 & 2.748 \\
0.50 & 2.212 & 2.349 & 2.511 & 2.686 \\
0.75 & 2.168 & 2.278 & 2.407 & 2.544 \\
\bottomrule
\end{tabular}
\end{table}

Figure~\ref{fig:rho} draws the rows of Table~\ref{tab:rho}. The balanced
compound-symmetry model stays on the curve \(\rho=1/2\) for every intra-class
correlation.

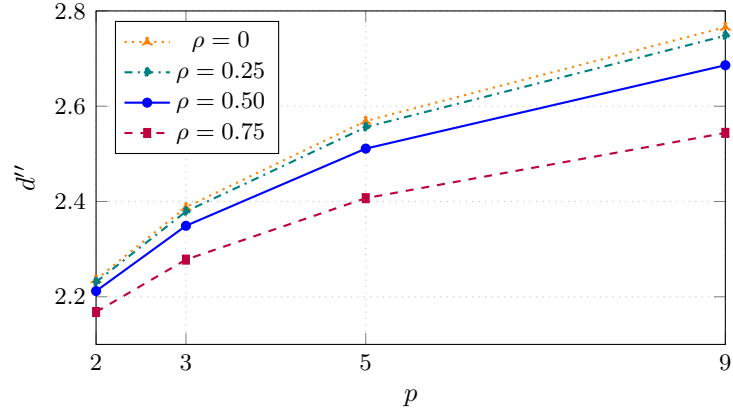
\begin{figure}[ht]
\centering
\begin{tikzpicture}
\begin{axis}[
  width=0.6\textwidth, height=6cm,
  xlabel={\(p\)}, ylabel={\(d''\)},
  xmin=2, xmax=9, ymin=2.1, ymax=2.8,
  xtick={2,3,5,9},
  legend pos=north west, legend style={font=\footnotesize},
  grid=major, grid style={dotted},
  tick label style={font=\footnotesize}, label style={font=\small},
]
\addplot[thick, orange, dotted, mark=triangle*, mark size=1.8pt] coordinates
  {(2,2.236) (3,2.388) (5,2.569) (9,2.766)};
\addlegendentry{\(\rho=0\)}
\addplot[thick, teal, dashdotted, mark=diamond*, mark size=1.8pt] coordinates
  {(2,2.231) (3,2.379) (5,2.556) (9,2.748)};
\addlegendentry{\(\rho=0.25\)}
\addplot[thick, blue, mark=*, mark size=1.5pt] coordinates
  {(2,2.212) (3,2.349) (5,2.511) (9,2.686)};
\addlegendentry{\(\rho=0.50\)}
\addplot[thick, purple, dashed, mark=square*, mark size=1.5pt] coordinates
  {(2,2.168) (3,2.278) (5,2.407) (9,2.544)};
\addlegendentry{\(\rho=0.75\)}
\end{axis}
\end{tikzpicture}
\caption{Solutions of~\eqref{eq:reduced} at \(\nu=\infty\) and \(P=0.95\) as
a function of \(p\), for the contrast correlations of Table~\ref{tab:rho}.}
\label{fig:rho}
\end{figure}

The correlation~\eqref{eq:tukey-corr} is the correlation of every family of
pairwise contrasts in which every pairwise difference has the same positive
variance.

\begin{proposition}
\label{prop:tukey}
Let \((Y_i)_{i\in I}\) have finite second moments, with \(I\) a finite index
set, and suppose \(\Var(Y_i-Y_j)=v>0\) whenever \(i\neq j\). Then, for all
\(i,j,a,b\in I\) with \(i\neq j\) and \(a\neq b\),
\begin{equation}
\label{eq:tukey-exact}
\begin{aligned}
\Corr(Y_i-Y_j,Y_a-Y_b)
&=
\frac12
\bigl(
\mathbf{1}_{\{i=a\}}+\mathbf{1}_{\{j=b\}}
-\mathbf{1}_{\{i=b\}}-\mathbf{1}_{\{j=a\}}
\bigr).
\end{aligned}
\end{equation}
\end{proposition}

\begin{proof}
For arbitrary indices, expand the variances by bilinearity:
\[
\Var(Y_u-Y_v)=\Var(Y_u)+\Var(Y_v)-2\Cov(Y_u,Y_v).
\]
The linear combination
\(\Var(Y_i-Y_b)+\Var(Y_j-Y_a)-\Var(Y_i-Y_a)-\Var(Y_j-Y_b)\)
therefore equals
\begin{align*}
-2\Cov(Y_i,Y_b)-2\Cov(Y_j,Y_a)+2\Cov(Y_i,Y_a)+2\Cov(Y_j,Y_b)
=
2\Cov(Y_i-Y_j,Y_a-Y_b).
\end{align*}
This is the identity
\begin{equation}
\label{eq:four-term}
2\Cov(Y_i-Y_j,Y_a-Y_b)
=
\Var(Y_i-Y_b)+\Var(Y_j-Y_a)
-\Var(Y_i-Y_a)-\Var(Y_j-Y_b),
\end{equation}
which uses only existence of the second moments. The diagonal case
\(u=v\) has variance \(0\), and the hypothesis gives variance \(v\) for
\(u\neq v\), so \(\Var(Y_u-Y_v)=v\bigl(1-\mathbf{1}_{\{u=v\}}\bigr)\).
Substitute these four variances into~\eqref{eq:four-term}:
\begin{align*}
2\Cov(Y_i-Y_j,Y_a-Y_b)
&=
v\Bigl[
\bigl(1-\mathbf{1}_{\{i=b\}}\bigr)
+\bigl(1-\mathbf{1}_{\{j=a\}}\bigr)
-\bigl(1-\mathbf{1}_{\{i=a\}}\bigr)
-\bigl(1-\mathbf{1}_{\{j=b\}}\bigr)
\Bigr]
\\
&=
v\bigl(
\mathbf{1}_{\{i=a\}}+\mathbf{1}_{\{j=b\}}
-\mathbf{1}_{\{i=b\}}-\mathbf{1}_{\{j=a\}}
\bigr).
\end{align*}
Both contrast variances equal \(v\) because \(i\neq j\) and \(a\neq b\), and
dividing the covariance by \(v\) produces~\eqref{eq:tukey-exact}.
\end{proof}

Equation~\eqref{eq:tukey-exact} is the closed form of~\eqref{eq:tukey-corr}.
Multiplication of every contrast by the positive constant \(1/\sqrt{v}\)
leaves the correlation unchanged, and the maximum absolute studentized
contrast has the law~\eqref{eq:tukey-student}.
The intra-class correlation does not appear in \(d_T\). The denominator is
the interaction mean square on \(\nu=(k-1)(n-1)\) degrees of freedom. With
\(k=p+1\) this is the \(\nu\) of~\eqref{eq:mse}, and the critical constant is
\(q_{P,k,\nu}/\sqrt{2}\). The ordinary analysis of variance has the same
contrast correlation and uses the residual mean square on \(k(n-1)\) degrees
of freedom. At a fixed \(\nu\) the two roots coincide.

\subsection{Reduction for the mixed model}
\label{sec:mm-reduce}

Integrals~\eqref{eq:6} and~\eqref{eq:7} are expectations of \(F\) and \(G\).
In the random-intercept model, those distribution functions are factorable due to the independence of the treatment residuals, once the control residual is given.

\subsubsection{Balanced random intercept and the multivariate-$t$ form}
\label{sec:mm-balanced}
\label{sec:mvt}

After the cancellation in~\eqref{eq:contrast-moments}, set
\begin{equation}
\label{eq:V}
V_i
=
\frac{\bar\varepsilon_{i\cdot}}{\sigma/\sqrt{n}},
\qquad i=0,\ldots,p,
\end{equation}
so that \(V_0,\ldots,V_p\) are independent standard normals, independent of
\(s\). Then
\begin{equation}
\label{eq:z-mixed}
\frac{z_i}{\sigma}
=
\frac{V_i-V_0}{\sqrt{2}},
\end{equation}
which is~\eqref{eq:factor} at \(\rho=1/2\) with \(U=-V_0\). Conditionally on
\(V_0=v\) and on \(w=s/\sigma\),
\begin{equation}
\label{eq:mm-cond-one}
F(d's,\ldots,d's)
=
\int_{-\infty}^{\infty}
\bigl[\Phi(v+d'w\sqrt{2})\bigr]^p
\phi(v)\,\d v.
\end{equation}
Both tails give
\begin{equation}
\label{eq:mm-cond-two}
G(d''s,\ldots,d''s)
=
\int_{-\infty}^{\infty}
\bigl[
\Phi(v+d''w\sqrt{2})-\Phi(v-d''w\sqrt{2})
\bigr]^p
\phi(v)\,\d v.
\end{equation}
Insert these expressions into~\eqref{eq:6} and~\eqref{eq:7}. The density
\(p(s)\,\d s\) becomes \(f_\nu(w)\,\d w\) with \(\nu=p(n-1)\), and the integrals
collapse to
\begin{equation}
\label{eq:mm-one}
P(d')
=
\int_0^{\infty}\int_{-\infty}^{\infty}
\bigl[\Phi(v+d'w\sqrt{2})\bigr]^p
\phi(v)\,f_\nu(w)\,\d v\,\d w
\end{equation}
and
\begin{equation}
\label{eq:mm-two}
P(d'')
=
\int_0^{\infty}\int_{-\infty}^{\infty}
\bigl[
\Phi(v+d''w\sqrt{2})-\Phi(v-d''w\sqrt{2})
\bigr]^p
\phi(v)\,f_\nu(w)\,\d v\,\d w.
\end{equation}
Equation~\eqref{eq:mm-one} is~\eqref{eq:reduced-one} at \(\rho=1/2\), and
\eqref{eq:mm-two} is~\eqref{eq:reduced} at \(\rho=1/2\).

The same integrals are rectangle probabilities for one multivariate \(t\)
vector, and that form carries the reduction beyond the balanced model to any
correlation that factors through a single control term.

Let \(D=(D_1,\ldots,D_p)\) be the vector of centered Dunnett contrasts, jointly
normal with
\[
\Var(D_i)=\sigma_*^2\,\Sigma_{ii},
\qquad
\Cov(D_i,D_k)=\sigma_*^2\,\Sigma_{ik},
\]
and let \(\nu s^2/\sigma_*^2\sim\chi^2_\nu\) be independent of \(D\). Define the
studentized coordinates by their own standard errors,
\begin{equation}
\label{eq:t-mvt}
t_i
=
\frac{D_i}{s\sqrt{\Sigma_{ii}}}.
\end{equation}
The numerators
\(z_i=D_i/\sqrt{\Sigma_{ii}}\) then satisfy \(\Var(z_i)=\sigma_*^2\) and
\[
\Corr(z_i,z_k)
=
R_{ik}
=
\frac{\Sigma_{ik}}{\sqrt{\Sigma_{ii}\,\Sigma_{kk}}},
\qquad
R_{ii}=1.
\]
The vector \(t\) has the central multivariate \(t\) distribution with \(\nu\)
degrees of freedom and correlation \(R\): it is a normal vector with
covariance \(\sigma_*^2 R\), divided by the independent scale \(s\). Dunnett's
integral~\eqref{eq:6} is the distribution function of this law at the
equicoordinate point \((d',\ldots,d')\). Integral~\eqref{eq:7} is the
probability of the cube \([-d'',d'']^p\). Writing \(w=s/\sigma_*\) and letting
\(Z\sim N(0,R)\),
\begin{equation}
\label{eq:mvt}
\begin{aligned}
P(d')
&=
\int_0^{\infty}
\Prob(Z_1<d'w,\ldots,Z_p<d'w)\,f_\nu(w)\,\d w,
\\
P(d'')
&=
\int_0^{\infty}
\Prob(|Z_1|<d''w,\ldots,|Z_p|<d''w)\,f_\nu(w)\,\d w.
\end{aligned}
\end{equation}
Equations~\eqref{eq:6} and~\eqref{eq:7} are~\eqref{eq:mvt}.

\paragraph{Product correlation.}
Suppose the margins at different levels of the Dunnett factor are independent
and the control margin is the only shared term, so
\(\Sigma_{ii}=v_i+v_0\) and \(\Sigma_{ik}=v_0\). Then
\begin{equation}
\label{eq:R-lambda}
R_{ik}=\lambda_i\lambda_k,
\qquad
\lambda_i=\sqrt{\frac{v_0}{v_0+v_i}}.
\end{equation}
The normal vector in~\eqref{eq:mvt} has the one-factor representation
\(Z_i=\lambda_i U+\sqrt{1-\lambda_i^2}\,V_i\) with \(U,V_i\) independent and
standard normal. Conditionally on \(U=u\), the coordinates separate and the
\(p\)-dimensional probability becomes a product of univariate terms. The
multivariate-\(t\) integral collapses to the double integral
\begin{equation}
\label{eq:mvt-lambda}
\begin{aligned}
P(d')
&=
\int_0^{\infty}\int_{-\infty}^{\infty}
\prod_{i=1}^{p}
\Phi\!\left(
\frac{d'w-\lambda_i u}{\sqrt{1-\lambda_i^2}}
\right)
\phi(u)\,f_\nu(w)\,\d u\,\d w,
\\[4pt]
P(d'')
&=
\int_0^{\infty}\int_{-\infty}^{\infty}
\prod_{i=1}^{p}
\Biggl[
\Phi\!\left(
\frac{d''w+\lambda_i u}{\sqrt{1-\lambda_i^2}}
\right)
-
\Phi\!\left(
\frac{-d''w+\lambda_i u}{\sqrt{1-\lambda_i^2}}
\right)
\Biggr]
\phi(u)\,f_\nu(w)\,\d u\,\d w.
\end{aligned}
\end{equation}
The one-sided line is~\eqref{eq:reduced-lambda}. Equal loadings
\(\lambda_i=\sqrt{1/2}\) return~\eqref{eq:mm-one} and~\eqref{eq:mm-two}.
The same double integral is the coverage probability whenever
\eqref{eq:R-lambda} holds, whether or not the \(v_i\) are equal.

\subsubsection{Unequal replication with the same cancellation}
\label{sec:mm-lambda}

If treatment \(i\) is replicated \(r_i\) times on every subject, the subject
effects still cancel. The loadings are
\(\lambda_i=(1+r_0/r_i)^{-1/2}\), as in~\eqref{eq:lambda}, and the integral
is~\eqref{eq:mvt-lambda}. The ratio \(\sigma_b^2/\sigma^2\) is absent.

\subsubsection{Conditioning on the random effects that remain}
\label{sec:mm-cond}

The passage from the \(p\)-dimensional probability \(F\) to the inner
integral of~\eqref{eq:mm-one} uses only the conditional independence of the
treatment residuals. The same step applies when a subject effect is still
present in the contrast, once that effect is placed in the conditioning
variable.

\begin{proposition}
\label{prop:cond}
Let \(\sigma>0\), and let \(z_1,\ldots,z_p\) and \(s>0\) be random variables.
Suppose that a random vector \(\xi\), measurable real functions \(a_i\),
constants \(c_i>0\), and random variables \(\eta_1,\ldots,\eta_p\) satisfy
\[
\frac{z_i}{\sigma}=a_i(\xi)+c_i\eta_i,
\qquad i=1,\ldots,p,
\]
that \(\eta_1,\ldots,\eta_p\) are independent standard normal variables, and
that \((\eta_1,\ldots,\eta_p)\) is independent of \((\xi,s)\). Set
\(w=s/\sigma\). Then, almost surely,
\begin{equation}
\label{eq:cond-general}
\begin{aligned}
&\Prob(z_1<d's,\ldots,z_p<d's\mid\xi,s) = \prod_{i=1}^{p} \Phi\!\left(\frac{d'w-a_i(\xi)}{c_i}\right),
\\[4pt]
&\Prob\bigl(|z_1|<d''s,\ldots,|z_p|<d''s\bigm|\xi,s\bigr) = \prod_{i=1}^{p} 
\Biggl[ 
\Phi\!\left(\frac{d''w-a_i(\xi)}{c_i}\right)
-
\Phi\!\left(\frac{-d''w-a_i(\xi)}{c_i}\right)
\Biggr].
\end{aligned}
\end{equation}
If \(s\) is also independent of \(\xi\) and has density \(p(s)\), then \(s\)
is independent of \((z_1,\ldots,z_p)\), and the integrals~\eqref{eq:6}
and~\eqref{eq:7} for this vector equal the expectations of the two products.
\end{proposition}

\begin{proof}
Independence of \((\eta_1,\ldots,\eta_p)\) and \((\xi,s)\) preserves the
independence and the standard normal laws after conditioning on \((\xi,s)\).
Because \(c_i>0\),
\[
\{z_i<d's\}
=
\left\{
\eta_i<\frac{d'w-a_i(\xi)}{c_i}
\right\}.
\]
The conditional probability of the \(i\)th event is the normal distribution
function in the first product, and conditional independence of the \(\eta_i\)
multiplies these probabilities. The same conditioning writes the two-sided
event as the interval
\[
\frac{-d''w-a_i(\xi)}{c_i}
<
\eta_i
<
\frac{d''w-a_i(\xi)}{c_i},
\]
whose probability is the difference of distribution functions in the second
product.

Each \(z_i\) is a function of \((\xi,\eta_i)\). For bounded measurable
functions \(f\), \(g\), and \(h\),
\begin{align*}
\E\bigl[f(s)\,g(\xi)\,h(\eta_1,\ldots,\eta_p)\bigr]
&=
\E\bigl[f(s)\,g(\xi)\bigr]\,\E\bigl[h(\eta_1,\ldots,\eta_p)\bigr],
\end{align*}
by independence of \((\eta_1,\ldots,\eta_p)\) and \((\xi,s)\). The further
assumption that \(s\) is independent of \(\xi\) factors
\(\E[f(s)g(\xi)]\), so \(s\) is independent of \((\xi,\eta_1,\ldots,\eta_p)\)
and hence of \((z_1,\ldots,z_p)\). Independence of \(s\) and
\((z_1,\ldots,z_p)\) gives
\[
\Prob(z_1<d's,\ldots,z_p<d's)
=
\int_{-\infty}^{\infty} F(d's,\ldots,d's)\,p(s)\,\d s,
\]
and the same identity with absolute values and with \(G\) in place of \(F\).
The law of total expectation equates the left-hand sides with the
expectations of the conditional probabilities in~\eqref{eq:cond-general}.
\end{proof}

In the balanced model, \(\xi=V_0\), \(\eta_i=V_i\),
\(a_i(\xi)=-V_0/\sqrt{2}\), and \(c_i=1/\sqrt{2}\). The first conditional
probability in~\eqref{eq:cond-general} equals
\(\bigl[\Phi(V_0+d'w\sqrt{2})\bigr]^p\), and its expectation in \(V_0\)
is~\eqref{eq:mm-cond-one}. The expectation of the second product
is~\eqref{eq:mm-cond-two}. The proposition is the reduction for any mixed
model whose contrasts, once the retained random effects are fixed, are
independent residual contrasts.

The split-plot simple effects of the between-subjects factor, taken at all
\(q\) levels of the repeated factor, are the case in which \(\xi\) is no
longer a single control residual. With equal group size \(n\) and one
observation per cell, put \(\tau^2=\sigma_s^2+\sigma^2\) and
\(\rho_Y=\sigma_s^2/\tau^2\), and let \(U_i\) and \(V_{ij}\) be independent
standard normals. The cell means admit the representation
\[
\frac{\bar Y_{ij\cdot}-\mu_{ij}}{\tau/\sqrt{n}}
=
\sqrt{\rho_Y}\,U_i+\sqrt{1-\rho_Y}\,V_{ij}.
\]
The Dunnett contrast at treatment \(i\) and level \(j\), divided by its
standard error \(\tau\sqrt{2/n}\), is \(C_{ij}/\sqrt{2}\) with
\[
C_{ij}
=
\sqrt{\rho_Y}\,(U_i-U_0)+\sqrt{1-\rho_Y}\,(V_{ij}-V_{0j}).
\]
The variables that are not free residuals are
\(\xi=(U_0,\ldots,U_p,V_{01},\ldots,V_{0q})\). Conditionally on \(\xi\), each
\(V_{ij}\) for \(i\ge 1\) appears in only one contrast, and
Proposition~\ref{prop:cond} gives a product of \(pq\) univariate factors.
For known variance components, \(w=1\) and integral~\eqref{eq:7} reduces to
\begin{equation}
\label{eq:mm-cross}
P(d'')
=
\int
\prod_{i=1}^{p}\prod_{j=1}^{q}
\Bigl[
\Phi\!\left(\alpha_{ij}+\frac{d''\sqrt{2}}{\sqrt{1-\rho_Y}}\right)
-
\Phi\!\left(\alpha_{ij}-\frac{d''\sqrt{2}}{\sqrt{1-\rho_Y}}\right)
\Bigr]
\phi(u)\,\phi(v_0)\,\d u\,\d v_0,
\end{equation}
where the integral runs over \(u\in\mathbb{R}^{p+1}\) and
\(v_0\in\mathbb{R}^{q}\), and
\[
\alpha_{ij}
=
V_{0j}
-
\frac{\sqrt{\rho_Y}\,(U_i-U_0)}{\sqrt{1-\rho_Y}}.
\]
The one-sided integral~\eqref{eq:6} keeps only the upper \(\Phi\)
in each factor. The dimension of the integral is \(p+1+q\), one variable for
each group mean and one for each control residual, rather than the dimension
\(pq\) of \(F\) and \(G\). At \(q=1\) the vector of contrasts is equicorrelated
with correlation \(1/2\) for every \(\rho_Y\), by Proposition~\ref{prop:cs},
and~\eqref{eq:mm-cross} collapses further to~\eqref{eq:mm-two}.

When the variance components are estimated, the exact chi-square denominator of the balanced model is replaced by a Satterthwaite approximation, and~\eqref{eq:mm-cross} is evaluated at \(d''w\) with that degrees of freedom. The factorisation \eqref{eq:cond-general} is unchanged: only the law of \(w\), and the value of \(\rho_Y\) inside \(\alpha_{ij}\), become approximate.

\subsection{The unbalanced case}
\label{sec:unbalanced}

Balance in model~\eqref{eq:lmm} means that every subject is observed once
under the control and once under each treatment. Two different departures
from that pattern have to be kept apart, because only one of them lets
\(\sigma_b^2\) into the correlation.

\subsubsection{Unequal replication when blocks cancel}
\label{sec:alloc}

Suppose subject \(j=1,\ldots,n\) is still observed under every treatment, but
treatment \(i\) is replicated \(r_i\) times within the subject. The cell mean
on subject \(j\) is \(\bar Y_{ij\cdot}=\mu_i+b_j+\bar\varepsilon_{ij\cdot}\),
with \(\Var(\bar\varepsilon_{ij\cdot})=\sigma^2/r_i\). Averaging over the
\(n\) subjects and forming \(D_i=\bar Y_{i\cdot\cdot}-\bar Y_{0\cdot\cdot}\),
the subject effects cancel exactly as in~\eqref{eq:contrast-moments}. What
remains is
\begin{equation}
\label{eq:var-unbal}
\Var(D_i)=\frac{\sigma^2}{n}\Bigl(\frac{1}{r_i}+\frac{1}{r_0}\Bigr),
\qquad
\Cov(D_i,D_k)=\frac{\sigma^2}{n r_0}.
\end{equation}
Hence
\begin{equation}
\label{eq:lambda}
\lambda_i
=
\sqrt{\frac{1/r_0}{1/r_i+1/r_0}}
=
\bigl(1+r_0/r_i\bigr)^{-1/2},
\qquad
\rho_{ik}=\lambda_i\lambda_k.
\end{equation}
The ratio \(\sigma_b^2/\sigma^2\) does not appear. The same loadings are
Dunnett's formula~\eqref{eq:rho-oneway} for independent groups. The one-sided
probability is
\begin{equation}
\label{eq:reduced-lambda}
P(d')
=
\int_0^{\infty}\int_{-\infty}^{\infty}
\prod_{i=1}^{p}
\Phi\!\left(
\frac{d'w-\lambda_i u}{\sqrt{1-\lambda_i^2}}
\right)
\phi(u)\,f_\nu(w)\,\d u\,\d w,
\end{equation}
and the two-sided probability replaces each factor by the difference of
\(\Phi\) values in~\eqref{eq:cond-two}. Equal replication returns
\(\lambda_i=1/\sqrt{2}\). In an independent-group layout the degrees of
freedom are \(\sum n_i-(p+1)\).

If \(r_1=\cdots=r_p=r\) and only the control replication differs, then
\(\rho=1/(1+r_0/r)\) and the matrix is equicorrelated.
Table~\ref{tab:alloc} solves the two-sided integral at \(P=0.95\) and
\(\nu=\infty\). Dunnett's recommendation
\(r_0/r\approx\sqrt{p}\)~\cite[Section~V]{Dunnett1955}, repeated
in~\cite{Dunnett1964}, is the row \(\rho=1/(1+\sqrt{p})\).

\begin{table}[ht]
\centering
\caption{Two-sided \(d''\) at \(P=0.95\) and \(\nu=\infty\), when every
treatment has the same replication and the control has \(r_0/r\) times as
many. Then \(\rho=1/(1+r_0/r)\).}
\label{tab:alloc}
\begin{tabular}{crrrrr}
\toprule
\(r_0/r\) & \(\rho\) & \(p=2\) & \(p=3\) & \(p=5\) & \(p=9\) \\
\midrule
1 & 0.500 & 2.212 & 2.349 & 2.511 & 2.686 \\
\(\sqrt{3}\) & 0.366 & 2.224 & 2.368 & 2.540 & 2.726 \\
2 & 0.333 & 2.227 & 2.372 & 2.546 & 2.734 \\
4 & 0.200 & 2.233 & 2.382 & 2.561 & 2.755 \\
\bottomrule
\end{tabular}
\end{table}

When the treatment replications differ, the \(\lambda_i\) differ. For
independent groups of sizes \((n_0,n_1,n_2,n_3)=(20,20,10,5)\),
\[
\lambda=(0.707,\,0.577,\,0.447)
\]
and the two-sided root at \(P=0.95\), \(\nu=\infty\) is \(2.372\), against
\(2.349\) at \(\rho=1/2\) and \(2.388\) at \(\rho=0\). Using \(1/2\) when the
control is more heavily replicated makes \(d''\) slightly too small.

\subsubsection{Incomplete subject sets}
\label{sec:incomplete}

The variance ratio \(\gamma=\sigma_b^2/\sigma^2\) enters the correlation when
different subjects are observed on different sets of treatments, because a
subject effect then cannot be removed from every comparison. Generalized
least squares with the compound-symmetry covariance of each observed set
produces the Dunnett correlation from \(L(X'V^{-1}X)^{-1}L'\).

If the \(p\) treatments are exchangeable, that correlation is a single
\(\rho\) depending on \(\gamma\). At \(\gamma=0\) it is the enlarged-control
value~\eqref{eq:lambda}. As \(\gamma\to\infty\), a control-only subject adds
nothing to a within-subject comparison and \(\rho\) returns to \(1/2\).

For \(p=2\) a one-factor representation always exists. For \(p\ge 3\) it need
not: no single vector \((\lambda_i)\) need reproduce the correlation, so
neither~\eqref{eq:reduced} nor~\eqref{eq:reduced-lambda} is the probability.
The integral is still~\eqref{eq:6} or~\eqref{eq:7}, computed from the full
correlation by methods of Genz and Bretz~\cite{GenzBretz2009} or Edwards and Berry~\cite{EdwardsBerry1987}. 

When \(\gamma\) and \(\sigma^2\) are replaced by restricted-maximum-likelihood
estimates, the denominator degrees of freedom are the Kenward-Roger
approximation~\cite{KenwardRoger1997}, or a Satterthwaite
approximation~\cite{Satterthwaite1946} when only one variance component is
estimated. Integrals~\eqref{eq:6} and~\eqref{eq:7} are evaluated at those
estimated correlations and degrees of freedom. Hothorn, Bretz, and
Westfall~\cite{HothornBretzWestfall2008} obtain these critical values from
the multivariate \(t\) with the estimated correlation of the comparisons in a
general parametric model. The R package \texttt{multcomp} applies that
procedure to linear mixed models.

\subsection{Other covariance structures}
\label{sec:cov}

The argument of \(F\) and \(G\) is the correlation of the Dunnett contrasts.
A covariance model changes integrals~\eqref{eq:6} and~\eqref{eq:7} when it
changes that correlation, or when it removes the single chi-square scale.
Jennrich and Schluchter~\cite{JennrichSchluchter1986} and
Wolfinger~\cite{Wolfinger1996} list the structures in ordinary use for
repeated measures: compound symmetry, heterogeneous compound symmetry,
autoregression, Toeplitz and ante-dependence forms, random coefficients,
factor-analytic residuals, and an unstructured matrix. In every case
Dunnett's integrals are~\eqref{eq:mvt} at the correlation \(R\) of the
contrasts. The reduction replaces the \(p\)-dimensional probability inside
that integral by an average, under Proposition~\ref{prop:cond}, of a product
of univariate normal probabilities. The dimension of the average is the
number of latent variables the contrasts share.

\subsubsection{A covariance shared by every treatment}
\label{sec:cov-common}

Comparisons across subjects are a different family from comparisons among
repeated measures on one subject. Suppose subjects are independent, every
treatment is observed on the same within-subject design, and the Dunnett
statistic is built from one common linear summary \(T_i=c'\bar Y_i\) of that
design. The summary can be the mean over the repeated levels, or any other
fixed contrast in those levels.

\begin{proposition}
\label{prop:common}
Let \(T_0,\ldots,T_p\) be pairwise uncorrelated, with \(\Var(T_i)=\tau^2>0\)
for every \(i\). Then \(\Var(T_i-T_0)=2\tau^2\) for every \(i=1,\ldots,p\),
and, for all \(1\le i\neq k\le p\),
\[
\Cov(T_i-T_0,T_k-T_0)=\tau^2,
\qquad
\Corr(T_i-T_0,T_k-T_0)=\frac12.
\]
\end{proposition}

\begin{proof}
Pairwise uncorrelatedness means \(\Cov(T_i,T_k)=0\) for \(i\neq k\). The
vector \((T_0,\ldots,T_p)\) therefore satisfies the hypotheses of
Proposition~\ref{prop:cs} with intra-class correlation \(\rho_Y=0\), which
lies in \([-1/p,1)\). Substituting \(\rho_Y=0\) into the moments of
Proposition~\ref{prop:cs} gives \(\Var(T_i-T_0)=2\tau^2\),
\(\Cov(T_i-T_0,T_k-T_0)=\tau^2\), and correlation \(1/2\). 
\end{proof}

The Proposition~\ref{prop:common} holds, in particular, when \(T_0,\ldots,T_p\) are independent.

Autoregression, a Toeplitz matrix, an unstructured matrix, a random slope,
and a Kronecker product all fall under this proposition when they are the
within-subject covariance of every treatment. The critical values are the
entries of Tables~\ref{tab:d95one} and~\ref{tab:d95}. The covariance enters the standard error
\(\tau\), and it enters the degrees of freedom only through the estimator of
\(\tau^2\). A single variance component estimated by one mean square keeps
\(w=s/\tau\) on the density~\eqref{eq:wdensity}. Integrals~\eqref{eq:6}
and~\eqref{eq:7} therefore reduce to~\eqref{eq:mm-one} and~\eqref{eq:mm-two}:
the contrast correlation is \(1/2\), and the covariance survives only as the
scale \(\tau\) inside \(w\). Two or more estimated components replace
\(f_\nu\) by a Satterthwaite or Kenward-Roger approximation.
The reduced integrands are unchanged.

The remainder of this section is the other family: Dunnett comparisons among
the repeated levels themselves, with level \(0\) as the control, or
independent groups whose variances differ.

\subsubsection{Heterogeneous compound symmetry}
\label{sec:cov-csh}

Heterogeneous compound symmetry (CSH) gives level \(j\) its own marginal standard
deviation \(\sigma_j\) and gives every pair the same correlation \(\rho\):
\begin{equation}
\label{eq:csh-cov}
\Cov(Y_j,Y_k)
=
\begin{cases}
\sigma_j^2, & j=k,\\
\rho\,\sigma_j\sigma_k, & j\neq k.
\end{cases}
\end{equation}
The representation
\begin{equation}
\label{eq:csh-rep}
Y_j
=
\mu_j+\sigma_j\bigl(\sqrt{\rho}\,U+\sqrt{1-\rho}\,V_j\bigr),
\end{equation}
with \(U,V_0,V_1,\ldots\) independent standard normals, reproduces
\eqref{eq:csh-cov}. The baseline contrasts \(D_i=Y_i-Y_0\) then have
\begin{equation}
\label{eq:csh-moments}
\begin{aligned}
\Var(D_i)
&=
\sigma_i^2+\sigma_0^2-2\rho\,\sigma_i\sigma_0,
\\
\Cov(D_i,D_k)
&=
\rho(\sigma_i-\sigma_0)(\sigma_k-\sigma_0)+(1-\rho)\sigma_0^2.
\end{aligned}
\end{equation}
Equal marginal variances reduce~\eqref{eq:csh-moments} to correlation \(1/2\),
for every \(\rho\), which is Proposition~\ref{prop:sphere}. Unequal variances
leave a correlation that depends on the whole list \((\sigma_j)\) and on
\(\rho\).

The shared variables in~\eqref{eq:csh-rep} are \(U\) and the control residual
\(V_0\). Conditionally on \((U,V_0)=(u,v)\), the contrasts are independent,
the conditional mean of \(D_i\) is
\(\sqrt{\rho}\,(\sigma_i-\sigma_0)u-\sigma_0\sqrt{1-\rho}\,v\),
and the conditional standard deviation is \(\sigma_i\sqrt{1-\rho}\).
Proposition~\ref{prop:cond} therefore supplies a product of univariate normal
probabilities. For known \((\sigma_j,\rho)\), the two-sided probability is
\begin{equation}
\label{eq:csh-two}
P(d'') =
\int_{-\infty}^{\infty}\int_{-\infty}^{\infty}
\prod_{i=1}^{p}
\Bigl[
\Phi\!\left(\frac{d''s_i-m_i(u,v)}{\sigma_i\sqrt{1-\rho}}\right)
-
\Phi\!\left(\frac{-d''s_i-m_i(u,v)}{\sigma_i\sqrt{1-\rho}}\right)
\Bigr]
\phi(u)\phi(v)\d u\d v,
\end{equation}
where \(s_i=\sqrt{\Var(D_i)}\) and
\(m_i(u,v)=\sqrt{\rho}\,(\sigma_i-\sigma_0)u-\sigma_0\sqrt{1-\rho}\,v\).
The one-sided integral is the same average with only the upper tail kept,
\begin{equation}
\label{eq:csh-one}
P(d')
=
\int_{-\infty}^{\infty}\int_{-\infty}^{\infty}
\prod_{i=1}^{p}
\Phi\!\left(\frac{d's_i-m_i(u,v)}{\sigma_i\sqrt{1-\rho}}\right)
\phi(u)\,\phi(v)\,\d u\,\d v.
\end{equation}
Before this conditioning, both probabilities are the \(p\)-dimensional cube
probabilities in~\eqref{eq:mvt} at the correlation read from
\eqref{eq:csh-moments}, with \(w=1\). The conditioning replaces that cube by
an integral of dimension two.

If the variance ratios \(a_j=\sigma_j/\sigma\) and the correlation \(\rho\)
are known, then \(\Sigma=\sigma^2 A\) with \(A\) known. An estimator \(s\)
independent of the contrasts and distributed as \(\sigma\sqrt{\chi^2_\nu/\nu}\)
restores the outer factor of Dunnett's integrals. Replace \(d'\) by \(d'w\)
and \(d''\) by \(d''w\), and mix~\eqref{eq:csh-one} and~\eqref{eq:csh-two}
against \(f_\nu(w)\):
\begin{equation}
\label{eq:csh-w}
P(d')
= \int_0^{\infty} P(d'w\mid a,\rho)\,f_\nu(w)\,\d w,
\quad
P(d'')
=
\int_0^{\infty} P(d''w\mid a,\rho)\,f_\nu(w)\,\d w.
\end{equation}
The symbols \(P(d'w\mid a,\rho)\) and \(P(d''w\mid a,\rho)\) are
\eqref{eq:csh-one} and~\eqref{eq:csh-two} evaluated at those scaled critical
values, with every \(\sigma_j\) replaced by the known ratio \(a_j\).
A fit in which \((a_j,\rho)\) are estimated does not leave this chi-square
scale. The same product is then evaluated at the fitted parameters, with a
Kenward-Roger or Satterthwaite law in place of \(f_\nu\).

A different heteroscedastic model keeps a common subject effect and gives the
residual at level \(j\) variance \(\psi_j\). The subject effect cancels in
\(D_i\). Set \(U=-V_0\) and
\[
\frac{D_i}{\sqrt{\psi_i+\psi_0}}
=
\lambda_i U+\sqrt{1-\lambda_i^2}\,V_i,
\qquad
\lambda_i=\sqrt{\frac{\psi_0}{\psi_0+\psi_i}}.
\]
The ratio \(\sigma_b^2/\sigma^2\) is absent. Integrals~\eqref{eq:6}
and~\eqref{eq:7} therefore reduce exactly to the one-factor
integrals~\eqref{eq:mvt-lambda}, with these loadings and with \(w=s/\sigma\)
on \(\nu\) degrees of freedom from the residual mean square of the \(\psi_j\)
when those ratios are known. The constant-correlation model~\eqref{eq:csh-cov}
is the one that requires the two latent variables in~\eqref{eq:csh-one}
and~\eqref{eq:csh-two}.

\subsubsection{Autoregression, Toeplitz structure, and ante-dependence}
\label{sec:cov-ar}

On equally spaced levels, stationary AR(1) has
\(\Corr(Y_j,Y_k)=\rho^{|j-k|}\) and a common variance \(\sigma^2\). The
baseline contrasts have
\begin{equation}
\label{eq:ar-cov}
\Var(D_i)
=
2\sigma^2(1-\rho^i),
\qquad
\Cov(D_i,D_k)
=
\sigma^2\bigl(1-\rho^i-\rho^k+\rho^{|i-k|}\bigr).
\end{equation}
For two follow-up levels the correlation simplifies to
\begin{equation}
\label{eq:ar-two}
\Corr(D_1,D_2)=\rho_D=\frac{\sqrt{1+\rho}}{2}.
\end{equation}
The pair is equicorrelated, so the \(p\)-dimensional factor in~\eqref{eq:mvt}
reduces to one normal integral. With \(\rho\) known, \(\Sigma=\sigma^2 A(\rho)\)
and an independent scale estimator \(s=\sigma w\) puts \(f_\nu\) on the outside.
Dunnett's integrals collapse to~\eqref{eq:reduced-one} and~\eqref{eq:reduced}
at this \(\rho_D\):
\begin{equation}
\label{eq:ar-reduced}
\begin{aligned}
P(d')
&=
\int_0^{\infty}\int_{-\infty}^{\infty}
\left[
\Phi\!\left(\frac{d'w-\sqrt{\rho_D}\,u}{\sqrt{1-\rho_D}}\right)
\right]^2
\phi(u)\,f_\nu(w)\,\d u\,\d w,
\\
P(d'')
&=
\int_0^{\infty}\int_{-\infty}^{\infty}
\left[
\Phi\!\left(\frac{d''w+\sqrt{\rho_D}\,u}{\sqrt{1-\rho_D}}\right)
-
\Phi\!\left(\frac{-d''w+\sqrt{\rho_D}\,u}{\sqrt{1-\rho_D}}\right)
\right]^2
\phi(u) f_\nu(w)\,\d u\,\d w.
\end{aligned}
\end{equation}
For three follow-up levels the same conditioning produces~\eqref{eq:mvt-lambda}
whenever the loadings defined by \(R_{ik}=\lambda_i\lambda_k\) satisfy
\(\lambda_i^2\le 1\). For four or more follow-up levels those loadings need
not exist. Integrals~\eqref{eq:6} and~\eqref{eq:7} then remain~\eqref{eq:mvt}
at~\eqref{eq:ar-cov}: given the baseline, the later levels are still
autoregressive, and the inner probability stays of dimension \(p\).

A Toeplitz covariance keeps equal variances and lets the lag correlations
\(\gamma_0,\gamma_1,\ldots\) be free, with \(\gamma_0=\sigma^2\).
The Dunnett integrals are again~\eqref{eq:mvt}, with~\eqref{eq:ar-cov}
evaluated at \(\rho^{|i-k|}\) replaced by \(\gamma_{|i-k|}/\gamma_0\).
AR(1) is the geometric case. The reduction is the same size as in the
autoregression: two follow-up levels reduce to~\eqref{eq:ar-reduced} at the
single lag correlation of the contrasts, and three follow-up levels reduce
to~\eqref{eq:mvt-lambda} whenever \(\lambda_i^2\le 1\).

Ante-dependence of order \(1\), introduced by Gabriel~\cite{Gabriel1962},
keeps the Markov property and drops stationarity. With consecutive
correlations \(\rho_1,\ldots,\rho_p\) and equal variances,
\[
\Corr(Y_j,Y_k)
=
\prod_{t=\min(j,k)+1}^{\max(j,k)}\rho_t
\]
for \(j\neq k\). AR(1) is the case of a constant \(\rho_t\). Substituting
this correlation for \(\rho^{|j-k|}\) in the expansion that produced
\eqref{eq:ar-cov} gives the contrast correlation, and~\eqref{eq:6}
and~\eqref{eq:7} are~\eqref{eq:mvt} at that matrix. They reduce to
\eqref{eq:ar-reduced} or to~\eqref{eq:mvt-lambda} under the same conditions
on \(p\) and on the loadings. Heterogeneous AR(1) uses the covariance
\(\sigma_j\sigma_k\rho^{|j-k|}\). Equal variances return~\eqref{eq:ar-cov};
a correlation that does not depend on lag returns~\eqref{eq:csh-moments}
and the two-dimensional integrals~\eqref{eq:csh-one} and~\eqref{eq:csh-two}.

\subsubsection{Random slope and a factor-analytic residual}
\label{sec:cov-slope}

A random intercept and an independent random slope,
\[
Y_j=\mu_j+b_0+b_1 t_j+\varepsilon_j,
\]
with \(\Var(b_1)=\gamma\sigma^2\), \(\Var(\varepsilon_j)=\sigma^2\), and
times \(t_0=0\), \(t_i=i\), give baseline contrasts in which \(b_0\) cancels.
In standardized form
\begin{equation}
\label{eq:rs-z}
\frac{D_i}{\sigma}
=
i\sqrt{\gamma}\,U+V_i-V_0,
\end{equation}
where \(U,V_0,V_1,\ldots\) are independent standard normals. Therefore
\begin{equation}
\label{eq:rs-moments}
\Var(D_i)=\sigma^2(i^2\gamma+2),
\qquad
\Cov(D_i,D_k)=\sigma^2(ik\gamma+1).
\end{equation}
At \(\gamma=0\) the correlation is \(1/2\). The free residual in each contrast
is \(V_i\), so Proposition~\ref{prop:cond} conditions on \(\xi=(U,V_0)\).
Before that step, integrals~\eqref{eq:6} and~\eqref{eq:7} are~\eqref{eq:mvt}
at the correlation~\eqref{eq:rs-moments}. After it, the \(p\)-dimensional
probability is a product of univariate terms. If \(\gamma\) is known, a
residual mean square for \(\sigma^2\) supplies \(w=s/\sigma\) with density
\(f_\nu\), and the integrals reduce to
\begin{equation}
\label{eq:rs-one}
P(d')
=
\int_0^{\infty}\int_{-\infty}^{\infty}\int_{-\infty}^{\infty}
\prod_{i=1}^{p}
\Phi(d'w\,s_i-i\sqrt{\gamma}\,u+v) \phi(u) \phi(v)\,f_\nu(w)\,\d u\,\d v\,\d w
\end{equation}
and
\begin{equation}
\label{eq:rs-two}
\begin{aligned}
P(d'')
&=
\int_0^{\infty}\int_{-\infty}^{\infty}\int_{-\infty}^{\infty}
\prod_{i=1}^{p}
\Bigl[
\Phi(d''w\,s_i-i\sqrt{\gamma}\,u+v)
-
\Phi(-d''w\,s_i-i\sqrt{\gamma}\,u+v)
\Bigr]
\\
&\qquad\times
\phi(u)\,\phi(v)\,f_\nu(w)\,\d u\,\d v\,\d w,
\end{aligned}
\end{equation}
with \(s_i=\sqrt{i^2\gamma+2}\). The dimension is three for every \(p\): two
latent normals and the scale \(w\). As \(\nu\to\infty\), \(w\) concentrates
at \(1\) and the integral over \(f_\nu\) drops out. If \(\gamma\) is
estimated, \(\gamma\) and \(\nu\) are plugged in from the mixed-model fit and
\(f_\nu\) is replaced by the Satterthwaite or Kenward-Roger approximation.
For four or more follow-up times the product form~\eqref{eq:mvt-lambda} need
not hold, while~\eqref{eq:rs-one} and~\eqref{eq:rs-two} remain integrals of
dimension three.

A factor-analytic residual of order \(1\),
\begin{equation}
\label{eq:fa-rep}
Y_j=\mu_j+\lambda_j U+\sqrt{\psi_j}\,V_j,
\end{equation}
with \(U,V_j\) independent and standard normal, has the same two conditioning
variables once the control residual is included. The baseline contrast is
\[
D_i=(\lambda_i-\lambda_0)U+\sqrt{\psi_i}\,V_i-\sqrt{\psi_0}\,V_0,
\]
so
\[
\Var(D_i)=(\lambda_i-\lambda_0)^2+\psi_i+\psi_0,
\qquad
\Cov(D_i,D_k)=(\lambda_i-\lambda_0)(\lambda_k-\lambda_0)+\psi_0.
\]
Integrals~\eqref{eq:6} and~\eqref{eq:7} are~\eqref{eq:mvt} at this
correlation. Conditionally on \((U,V_0)=(u,v)\) the contrasts are
independent, with conditional mean
\(m_i(u,v)=(\lambda_i-\lambda_0)u-\sqrt{\psi_0}\,v\) and conditional standard
deviation \(\sqrt{\psi_i}\). Proposition~\ref{prop:cond} reduces both integrals
to dimension two plus the scale. Writing \(s_i=\sqrt{\Var(D_i)}\) and
\(w=s/\sigma\) when \(\lambda_j=\sigma\alpha_j\) and \(\psi_j=\sigma^2\beta_j\)
with \((\alpha,\beta)\) known,
\begin{equation}
\label{eq:fa-two}
\begin{aligned}
P(d'')
&=
\int_0^{\infty}\int_{-\infty}^{\infty}\int_{-\infty}^{\infty}
\prod_{i=1}^{p}
\Bigl[
\Phi\!\left(\frac{d''w\,s_i-m_i(u,v)}{\sqrt{\psi_i}}\right)
-
\Phi\!\left(\frac{-d''w\,s_i-m_i(u,v)}{\sqrt{\psi_i}}\right)
\Bigr]
\\
&\qquad\times
\phi(u)\,\phi(v)\,f_\nu(w)\,\d u\,\d v\,\d w.
\end{aligned}
\end{equation}
The one-sided integral~\eqref{eq:6} keeps the upper \(\Phi\). Heterogeneous
compound symmetry is the substitution \(\lambda_j=\sigma_j\sqrt{\rho}\) and
\(\psi_j=\sigma_j^2(1-\rho)\), which returns~\eqref{eq:csh-one}
and~\eqref{eq:csh-two}. The random-slope representation~\eqref{eq:rs-z} is
the substitution \(\lambda_i=i\sigma\sqrt{\gamma}\), \(\lambda_0=0\), and
\(\psi_j=\sigma^2\), which returns~\eqref{eq:rs-one} and~\eqref{eq:rs-two}.
Jennrich and Schluchter~\cite{JennrichSchluchter1986} include this residual
among the structured covariances that can be fitted by likelihood. 

\subsubsection{Kronecker products}
\label{sec:cov-kron}

Two repeated factors on a common grid, with covariance
\(\Sigma_A\otimes\Sigma_B\) for the vectorized cells, give marginal means
across \(B\) a covariance proportional to \(\Sigma_A\). The proportionality
constant is the average quadratic form \(c_B=b^{-2}1'\Sigma_B 1\), where \(b\) is
the number of levels of \(B\). It
cancels in the correlation, so a marginal Dunnett family on \(A\) has the
contrast correlation induced by \(\Sigma_A\) alone. Integrals~\eqref{eq:6}
and~\eqref{eq:7} therefore reduce to the Dunnett integrals of \(\Sigma_A\).
If \(\Sigma_A\) is spherical, that correlation is \(1/2\) and the integrals
are~\eqref{eq:mm-one} and~\eqref{eq:mm-two}, with the contrast standard
deviation multiplied by \(\sqrt{c_B}\). If \(\Sigma_A\) is AR(1),
they are~\eqref{eq:ar-reduced} or~\eqref{eq:mvt} according to the number of
follow-up levels, as in Section~\ref{sec:cov-ar}.
Proposition~\ref{prop:common} is the further case in which the whole grid is
averaged and the Dunnett family compares independent subjects: the integrals
reduce again to~\eqref{eq:mm-one} and~\eqref{eq:mm-two}.

Cellwise comparisons of every treatment with the control, at every level of
\(B\), start as~\eqref{eq:mvt} in dimension \(pq\). Let the compound-symmetry
factor on the treatments, including the control, have off-diagonal correlation
\(\rho\), and let \(\Sigma_E\) be the covariance along \(B\). The representation
\[
Y_i=\sqrt{\rho}\,W+\sqrt{1-\rho}\,Z_i,
\]
with \(W,Z_0,Z_1,\ldots\) independent and each distributed as \(N(0,\Sigma_E)\),
reproduces the Kronecker product of that compound-symmetry matrix with
\(\Sigma_E\). In the contrast \(Y_i-Y_0\) the common
vector \(W\) cancels, and \(\rho\) cancels from the correlation. What remains is
\[
D_i=\sqrt{1-\rho}\,(Z_i-Z_0).
\]
The correlation is \(1/2\) between two treatments at the same level of \(B\),
the correlation of \(\Sigma_E\) between two levels of the same treatment, and
half that correlation between different treatments at different levels.
Conditionally on the control slice
\(Z_0\), the vectors \(Z_i\) are independent, so the \(pq\)-dimensional
probability factors into a product of \(p\) copies of a \(q\)-dimensional
normal probability. The outer integral has dimension \(q\). When \(\Sigma_E\)
is spherical, each of those factors splits into univariate \(\Phi\) terms and
the integral reduces further to the block product~\eqref{eq:mm-blocks}.
When \(\Sigma_E\) is autoregressive, the inner factor remains the
\(q\)-dimensional cube probability of that autoregression.

\subsubsection{An unstructured covariance}
\label{sec:cov-un}

An unstructured covariance places no constraint on \(\Sigma\) beyond positive
definiteness. For baseline contrasts,
$\Var(D_i)=\Sigma_{ii}+\Sigma_{00}-2\Sigma_{i0}$, $\Cov(D_i,D_k)=\Sigma_{ik}-\Sigma_{i0}-\Sigma_{k0}+\Sigma_{00}$,
and \(R\) is this covariance scaled to unit diagonal. There is no loading
identity. Integrals~\eqref{eq:6} and~\eqref{eq:7} are exactly~\eqref{eq:mvt}
at that \(R\) and at the degrees of freedom of the scale estimator.
The inner probability does not factor through
Proposition~\ref{prop:cond} for a general \(R\), so the dimension stays \(p\). Nevertheless, it can still be evaluated in a direct manner \cite{GenzBretz2009}.

\subsubsection{Heteroscedastic independent groups}
\label{sec:cov-het}

Independent groups with variances \(\sigma_i^2\) and sizes \(n_i\) have
$\Var(D_i)={\sigma_i^2}/{n_i}+{\sigma_0^2}/{n_0}$ and $\Cov(D_i,D_k)={\sigma_0^2}/{n_0}$.
The correlation is the product form~\eqref{eq:R-lambda} with
\begin{equation}
\label{eq:het-lambda}
\lambda_i
=
\sqrt{\frac{\sigma_0^2/n_0}{\sigma_0^2/n_0+\sigma_i^2/n_i}}.
\end{equation}
The control mean is the only factor shared by the contrasts. When the
variance ratios are known, studentization by the known relative standard
errors leaves a central multivariate \(t\) on the residual degrees of
freedom. Integrals~\eqref{eq:6} and~\eqref{eq:7} are~\eqref{eq:mvt} at
\(R_{ik}=\lambda_i\lambda_k\), and the conditioning of
Section~\ref{sec:mvt} reduces them to~\eqref{eq:mvt-lambda} at the loadings
\eqref{eq:het-lambda}. Equal variances return Dunnett's sample-size
formula~\eqref{eq:rho-oneway}, and equal sample sizes return \(\lambda_i=1/\sqrt{2}\)
and the integrals~\eqref{eq:reduced-one} and~\eqref{eq:reduced}.

Separate variance estimates \(s_i^2\) in each group replace the single scale
\(s\). The statistics
\[
{D_i}/{\sqrt{s_i^2/n_i+s_0^2/n_0}}
\]
are then not a central multivariate \(t\),
and integrals~\eqref{eq:6} and~\eqref{eq:7} are no longer the coverage
probability. Dunnett~\cite{Dunnett1980} compared several approximate
procedures for that pairwise problem.

\subsubsection{Covariance structures and the 1955 tables}
\label{sec:cov-1955}

The roots of~\eqref{eq:7} at \(\rho=1/2\) were compared above with the
constants of Dunnett's Table~2a. The same comparison for the structures of
this section asks whether those constants remain conservative once the
correlation moves away from \(1/2\). Each structure is
extended to \(p\) follow-up levels or treatments as follows. AR(1) uses
\(\rho=0.5\) and \(\rho=0.9\) on levels \(0,\ldots,p\). Heterogeneous compound
symmetry uses \(\rho=0.5\) and standard deviations \(\sigma_0=1\),
\(\sigma_j=1+j/4\). The random slope uses \(\gamma=1\) at times \(0,\ldots,p\).
The Toeplitz structure has lag correlations \(1,0.6,0.3,0.1\) and zero beyond
lag three. Ante-dependence alternates consecutive correlations \(0.8\) and
\(0.4\). The heteroscedastic groups have equal sizes and variances
\(\sigma_0^2=1\), \(\sigma_i^2=1+i/2\). Roots for \(p\ge 4\) that are not
one-factor were computed by the Genz-Bretz algorithm~\cite{GenzBretz2009},
with an absolute error below \(0.002\) in \(d''\). The one-factor entries
were checked against~\eqref{eq:mvt-lambda}.

Table~\ref{tab:cov1955} lists the roots at \(\nu=\infty\) and \(\nu=5\), and
Figures~\ref{fig:cov1955a} and~\ref{fig:cov1955b} plot them against the
Table~2a constants and the \(\rho=1/2\) roots.

\begin{table}[ht]
\centering
\caption{Two-sided roots \(d''\) of integral~\eqref{eq:7} at \(P=0.95\) for
\(p\) comparisons under each covariance structure, with the \(\rho=1/2\) roots
of Table~\ref{tab:d95} and the constants of Dunnett's Table~2a.}
\label{tab:cov1955}
\footnotesize
\setlength{\tabcolsep}{4pt}
\begin{tabular}{lrrrrrrrrr}
\toprule
& \multicolumn{9}{c}{\(p\)} \\
\cmidrule(lr){2-10}
Structure & 1 & 2 & 3 & 4 & 5 & 6 & 7 & 8 & 9 \\
\midrule
\multicolumn{10}{l}{\(\nu=\infty\)} \\
Table~2a (1955)            & 1.96  & 2.21  & 2.37  & 2.47  & 2.55  & 2.62  & 2.67  & 2.71  & 2.75 \\
\(\rho=1/2\)               & 1.960 & 2.212 & 2.349 & 2.442 & 2.511 & 2.567 & 2.613 & 2.652 & 2.686 \\
AR(1), \(\rho=0.5\)        & 1.960 & 2.197 & 2.325 & 2.413 & 2.479 & 2.533 & 2.579 & 2.617 & 2.651 \\
AR(1), \(\rho=0.9\)        & 1.960 & 2.182 & 2.298 & 2.373 & 2.429 & 2.473 & 2.510 & 2.540 & 2.567 \\
CSH, \(\sigma_j=1+j/4\)    & 1.960 & 2.224 & 2.370 & 2.470 & 2.546 & 2.607 & 2.657 & 2.699 & 2.736 \\
Random slope, \(\gamma=1\) & 1.960 & 2.178 & 2.269 & 2.314 & 2.339 & 2.353 & 2.362 & 2.368 & 2.372 \\
Toeplitz \(1,0.6,0.3,0.1\) & 1.960 & 2.188 & 2.310 & 2.395 & 2.461 & 2.515 & 2.560 & 2.599 & 2.634 \\
Ante-dependence \(0.8,0.4\)& 1.960 & 2.223 & 2.315 & 2.415 & 2.463 & 2.525 & 2.558 & 2.603 & 2.629 \\
Variances \(1+i/2\)        & 1.960 & 2.224 & 2.371 & 2.473 & 2.551 & 2.613 & 2.665 & 2.710 & 2.749 \\
\midrule
\multicolumn{10}{l}{\(\nu=5\)} \\
Table~2a (1955)            & 2.57  & 3.03  & 3.39  & 3.66  & 3.88  & 4.06  & 4.22  & 4.36  & 4.49 \\
\(\rho=1/2\)               & 2.571 & 3.030 & 3.293 & 3.476 & 3.615 & 3.727 & 3.821 & 3.900 & 3.970 \\
AR(1), \(\rho=0.5\)        & 2.571 & 2.995 & 3.237 & 3.407 & 3.539 & 3.646 & 3.737 & 3.814 & 3.884 \\
AR(1), \(\rho=0.9\)        & 2.571 & 2.964 & 3.177 & 3.320 & 3.428 & 3.513 & 3.583 & 3.644 & 3.697 \\
CSH, \(\sigma_j=1+j/4\)    & 2.571 & 3.059 & 3.348 & 3.553 & 3.712 & 3.841 & 3.948 & 4.039 & 4.119 \\
Random slope, \(\gamma=1\) & 2.571 & 2.956 & 3.118 & 3.196 & 3.240 & 3.265 & 3.281 & 3.291 & 3.299 \\
Toeplitz \(1,0.6,0.3,0.1\) & 2.571 & 2.976 & 3.205 & 3.366 & 3.497 & 3.605 & 3.696 & 3.775 & 3.845 \\
Ante-dependence \(0.8,0.4\)& 2.571 & 3.057 & 3.224 & 3.422 & 3.512 & 3.638 & 3.701 & 3.794 & 3.843 \\
Variances \(1+i/2\)        & 2.571 & 3.060 & 3.352 & 3.562 & 3.726 & 3.862 & 3.977 & 4.077 & 4.166 \\
\bottomrule
\end{tabular}
\end{table}

\begin{figure}[ht]
\centering
\begin{tikzpicture}
\begin{axis}[
  width=0.48\textwidth, height=6.2cm,
  title={\(\nu=\infty\)},
  xlabel={\(p\)}, ylabel={\(d''\)},
  xmin=1, xmax=9, ymin=1.9, ymax=2.8, xtick={1,...,9},
  legend pos=north west, legend style={font=\scriptsize, cells={anchor=west}},
  grid=major, grid style={dotted},
  tick label style={font=\footnotesize}, label style={font=\small},
]
\addplot[thick, red, dashed, mark=square*, mark size=1.5pt] coordinates
  {(1,1.96) (2,2.21) (3,2.37) (4,2.47) (5,2.55) (6,2.62) (7,2.67) (8,2.71) (9,2.75)};
\addlegendentry{Table~2a}
\addplot[thick, blue, mark=*, mark size=1.5pt] coordinates
  {(1,1.960) (2,2.212) (3,2.349) (4,2.442) (5,2.511) (6,2.567) (7,2.613) (8,2.652) (9,2.686)};
\addlegendentry{\(\rho=1/2\)}
\addplot[thick, orange, dotted, mark=triangle*, mark size=1.8pt] coordinates
  {(1,1.960) (2,2.197) (3,2.325) (4,2.413) (5,2.479) (6,2.533) (7,2.579) (8,2.617) (9,2.651)};
\addlegendentry{AR(1), \(\rho=0.5\)}
\addplot[thick, green!50!black, dashdotted, mark=triangle, mark size=1.8pt] coordinates
  {(1,1.960) (2,2.182) (3,2.298) (4,2.373) (5,2.429) (6,2.473) (7,2.510) (8,2.540) (9,2.567)};
\addlegendentry{AR(1), \(\rho=0.9\)}
\addplot[thick, purple, densely dotted, mark=diamond*, mark size=1.8pt] coordinates
  {(1,1.960) (2,2.224) (3,2.370) (4,2.470) (5,2.546) (6,2.607) (7,2.657) (8,2.699) (9,2.736)};
\addlegendentry{CSH}
\addplot[thick, brown, loosely dashed, mark=o, mark size=1.5pt] coordinates
  {(1,1.960) (2,2.178) (3,2.269) (4,2.314) (5,2.339) (6,2.353) (7,2.362) (8,2.368) (9,2.372)};
\addlegendentry{random slope}
\end{axis}
\end{tikzpicture}
\hfill
\begin{tikzpicture}
\begin{axis}[
  width=0.48\textwidth, height=6.2cm,
  title={\(\nu=5\)},
  xlabel={\(p\)}, ylabel={\(d''\)},
  xmin=1, xmax=9, ymin=2.4, ymax=4.7, xtick={1,...,9},
  legend pos=north west, legend style={font=\scriptsize, cells={anchor=west}},
  grid=major, grid style={dotted},
  tick label style={font=\footnotesize}, label style={font=\small},
]
\addplot[thick, red, dashed, mark=square*, mark size=1.5pt] coordinates
  {(1,2.57) (2,3.03) (3,3.39) (4,3.66) (5,3.88) (6,4.06) (7,4.22) (8,4.36) (9,4.49)};
\addlegendentry{Table~2a}
\addplot[thick, blue, mark=*, mark size=1.5pt] coordinates
  {(1,2.571) (2,3.030) (3,3.293) (4,3.476) (5,3.615) (6,3.727) (7,3.821) (8,3.900) (9,3.970)};
\addlegendentry{\(\rho=1/2\)}
\addplot[thick, orange, dotted, mark=triangle*, mark size=1.8pt] coordinates
  {(1,2.571) (2,2.995) (3,3.237) (4,3.407) (5,3.539) (6,3.646) (7,3.737) (8,3.814) (9,3.884)};
\addlegendentry{AR(1), \(\rho=0.5\)}
\addplot[thick, green!50!black, dashdotted, mark=triangle, mark size=1.8pt] coordinates
  {(1,2.571) (2,2.964) (3,3.177) (4,3.320) (5,3.428) (6,3.513) (7,3.583) (8,3.644) (9,3.697)};
\addlegendentry{AR(1), \(\rho=0.9\)}
\addplot[thick, purple, densely dotted, mark=diamond*, mark size=1.8pt] coordinates
  {(1,2.571) (2,3.059) (3,3.348) (4,3.553) (5,3.712) (6,3.841) (7,3.948) (8,4.039) (9,4.119)};
\addlegendentry{CSH}
\addplot[thick, brown, loosely dashed, mark=o, mark size=1.5pt] coordinates
  {(1,2.571) (2,2.956) (3,3.118) (4,3.196) (5,3.240) (6,3.265) (7,3.281) (8,3.291) (9,3.299)};
\addlegendentry{random slope}
\end{axis}
\end{tikzpicture}
\caption{Roots of integral~\eqref{eq:7} at \(P=0.95\) for baseline comparisons
under AR(1), heterogeneous compound symmetry, and a random slope, with
the \(\rho=1/2\) roots and the constants of Dunnett's Table~2a. Left,
\(\nu=\infty\); right, \(\nu=5\). }
\label{fig:cov1955a}
\end{figure}
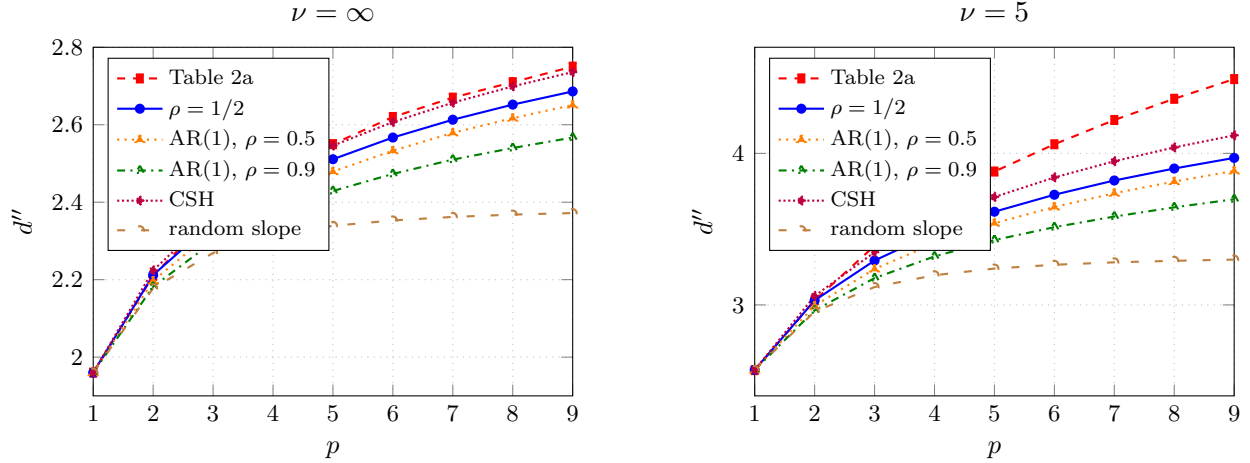

\begin{figure}[ht]
\centering
\begin{tikzpicture}
\begin{axis}[
  width=0.48\textwidth, height=6.2cm,
  title={\(\nu=\infty\)},
  xlabel={\(p\)}, ylabel={\(d''\)},
  xmin=1, xmax=9, ymin=1.9, ymax=2.8, xtick={1,...,9},
  legend pos=north west, legend style={font=\scriptsize, cells={anchor=west}},
  grid=major, grid style={dotted},
  tick label style={font=\footnotesize}, label style={font=\small},
]
\addplot[thick, red, dashed, mark=square*, mark size=1.5pt] coordinates
  {(1,1.96) (2,2.21) (3,2.37) (4,2.47) (5,2.55) (6,2.62) (7,2.67) (8,2.71) (9,2.75)};
\addlegendentry{Table~2a}
\addplot[thick, blue, mark=*, mark size=1.5pt] coordinates
  {(1,1.960) (2,2.212) (3,2.349) (4,2.442) (5,2.511) (6,2.567) (7,2.613) (8,2.652) (9,2.686)};
\addlegendentry{\(\rho=1/2\)}
\addplot[thick, orange, dotted, mark=triangle*, mark size=1.8pt] coordinates
  {(1,1.960) (2,2.188) (3,2.310) (4,2.395) (5,2.461) (6,2.515) (7,2.560) (8,2.599) (9,2.634)};
\addlegendentry{Toeplitz}
\addplot[thick, green!50!black, dashdotted, mark=triangle, mark size=1.8pt] coordinates
  {(1,1.960) (2,2.223) (3,2.315) (4,2.415) (5,2.463) (6,2.525) (7,2.558) (8,2.603) (9,2.629)};
\addlegendentry{ante-dependence}
\addplot[thick, purple, densely dotted, mark=diamond*, mark size=1.8pt] coordinates
  {(1,1.960) (2,2.224) (3,2.371) (4,2.473) (5,2.551) (6,2.613) (7,2.665) (8,2.710) (9,2.749)};
\addlegendentry{variances \(1+i/2\)}
\end{axis}
\end{tikzpicture}
\hfill
\begin{tikzpicture}
\begin{axis}[
  width=0.48\textwidth, height=6.2cm,
  title={\(\nu=5\)},
  xlabel={\(p\)}, ylabel={\(d''\)},
  xmin=1, xmax=9, ymin=2.4, ymax=4.7, xtick={1,...,9},
  legend pos=north west, legend style={font=\scriptsize, cells={anchor=west}},
  grid=major, grid style={dotted},
  tick label style={font=\footnotesize}, label style={font=\small},
]
\addplot[thick, red, dashed, mark=square*, mark size=1.5pt] coordinates
  {(1,2.57) (2,3.03) (3,3.39) (4,3.66) (5,3.88) (6,4.06) (7,4.22) (8,4.36) (9,4.49)};
\addlegendentry{Table~2a}
\addplot[thick, blue, mark=*, mark size=1.5pt] coordinates
  {(1,2.571) (2,3.030) (3,3.293) (4,3.476) (5,3.615) (6,3.727) (7,3.821) (8,3.900) (9,3.970)};
\addlegendentry{\(\rho=1/2\)}
\addplot[thick, orange, dotted, mark=triangle*, mark size=1.8pt] coordinates
  {(1,2.571) (2,2.976) (3,3.205) (4,3.366) (5,3.497) (6,3.605) (7,3.696) (8,3.775) (9,3.845)};
\addlegendentry{Toeplitz}
\addplot[thick, green!50!black, dashdotted, mark=triangle, mark size=1.8pt] coordinates
  {(1,2.571) (2,3.057) (3,3.224) (4,3.422) (5,3.512) (6,3.638) (7,3.701) (8,3.794) (9,3.843)};
\addlegendentry{ante-dependence}
\addplot[thick, purple, densely dotted, mark=diamond*, mark size=1.8pt] coordinates
  {(1,2.571) (2,3.060) (3,3.352) (4,3.562) (5,3.726) (6,3.862) (7,3.977) (8,4.077) (9,4.166)};
\addlegendentry{variances \(1+i/2\)}
\end{axis}
\end{tikzpicture}
\caption{Roots of integral~\eqref{eq:7} at \(P=0.95\) under the banded
Toeplitz structure, alternating ante-dependence, and heteroscedastic
independent groups, with the \(\rho=1/2\) roots and Dunnett's Table~2a. Left,
\(\nu=\infty\); right, \(\nu=5\). }
\label{fig:cov1955b}
\end{figure}
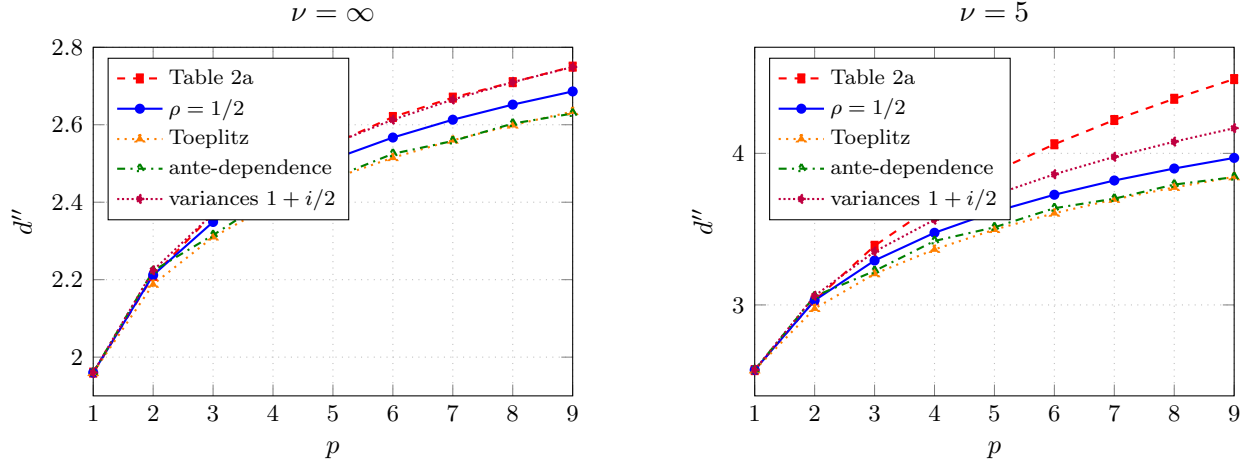

Structures whose contrast correlation exceeds \(1/2\), namely AR(1), the
Toeplitz band, alternating ante-dependence, and the random slope, have roots
below the \(\rho=1/2\) curve, so Table~2a is conservative for them. The
random slope is the largest gap. Heterogeneous compound symmetry and
heteroscedastic groups push the correlation below \(1/2\), and their roots
lie above the \(\rho=1/2\) curve. At \(p=2\), where the bound~\eqref{eq:11}
is an equality, Table~2a lies below those roots. From \(p=3\) the slack in
the bound overtakes the correlation effect. The \(\rho=1/2\) root stays
anticonservative for those two structures at every \(p\).

\subsubsection{Covariance structures that move the constant}
\label{sec:tukey-cov}

Once the pairwise variances differ, or a disjoint pair acquires a nonzero
correlation, \eqref{eq:pearson-hartley} ceases to be the probability.
Each contrast is then divided by its own standard deviation, and a single
critical constant \(d_T\) is applied to every studentized pair. That common
constant is the range convention when all standard deviations are equal, and
it is kept when they are not, as in Section~\ref{sec:integral}. The
probability is the multivariate-\(t\) cube at the resulting correlation,
mixed over \(f_\nu\) when a single chi-square scale is available. The cube
has dimension \(m=k(k-1)/2\) and rank \(k-1\).

The structures are those of Section~\ref{sec:cov-1955}, applied to all
\(k\) levels rather than to baseline contrasts against level \(0\).

Table~\ref{tab:tukey} gives \(d_T\) at \(P=0.95\). The first row is the ordinary and spherical root of~\eqref{eq:tukey-student}. The other rows are the pairwise cube under the covariances of Section~\ref{sec:cov-1955}.

\begin{table}[ht]
\centering
\caption{Critical values \(d_T=q/\sqrt{2}\) for Tukey's studentized range at
\(P=0.95\).}
\label{tab:tukey}
\footnotesize
\setlength{\tabcolsep}{4pt}
\begin{tabular}{lrrrrrr}
\toprule
& \multicolumn{6}{c}{\(k\)} \\
\cmidrule(lr){2-7}
Structure & 3 & 4 & 5 & 6 & 7 & 8 \\
\midrule
\multicolumn{7}{l}{\(\nu=\infty\)} \\
Studentized range          & 2.344 & 2.569 & 2.728 & 2.850 & 2.948 & 3.031 \\
AR(1), \(\rho=0.5\)        & 2.336 & 2.555 & 2.708 & 2.826 & 2.922 & 3.001 \\
AR(1), \(\rho=0.9\)        & 2.322 & 2.528 & 2.670 & 2.778 & 2.863 & 2.935 \\
CSH, \(\sigma_j=1+(j-1)/4\)& 2.340 & 2.561 & 2.715 & 2.833 & 2.926 & 3.004 \\
Random slope, \(\gamma=1\) & 2.317 & 2.500 & 2.613 & 2.691 & 2.750 & 2.797 \\
Toeplitz \(1,0.6,0.3,0.1\) & 2.328 & 2.540 & 2.689 & 2.803 & 2.896 & 2.975 \\
Ante-dependence \(0.8,0.4\)& 2.318 & 2.517 & 2.677 & 2.781 & 2.884 & 2.953 \\
Variances \(1+(i-1)/2\)    & 2.341 & 2.563 & 2.719 & 2.838 & 2.933 & 3.013 \\
\midrule
\multicolumn{7}{l}{\(\nu=5\)} \\
Studentized range          & 3.254 & 3.690 & 4.012 & 4.266 & 4.476 & 4.654 \\
AR(1), \(\rho=0.5\)        & 3.242 & 3.664 & 3.973 & 4.217 & 4.419 & 4.589 \\
AR(1), \(\rho=0.9\)        & 3.216 & 3.614 & 3.900 & 4.122 & 4.305 & 4.456 \\
CSH, \(\sigma_j=1+(j-1)/4\)& 3.248 & 3.676 & 3.989 & 4.234 & 4.434 & 4.605 \\
Random slope, \(\gamma=1\) & 3.208 & 3.559 & 3.783 & 3.941 & 4.063 & 4.163 \\
Toeplitz \(1,0.6,0.3,0.1\) & 3.227 & 3.637 & 3.935 & 4.169 & 4.365 & 4.533 \\
Ante-dependence \(0.8,0.4\)& 3.214 & 3.601 & 3.919 & 4.136 & 4.350 & 4.500 \\
Variances \(1+(i-1)/2\)    & 3.249 & 3.679 & 3.995 & 4.244 & 4.447 & 4.621 \\
\bottomrule
\end{tabular}
\end{table}

Every root in Table~\ref{tab:tukey} lies below the studentized range root,
so the ordinary Tukey constant is conservative for these covariances. The
gap is largest for the random slope. Figure~\ref{fig:tukey} plots the
ordinary root against the four structures that account for most of the
spread.

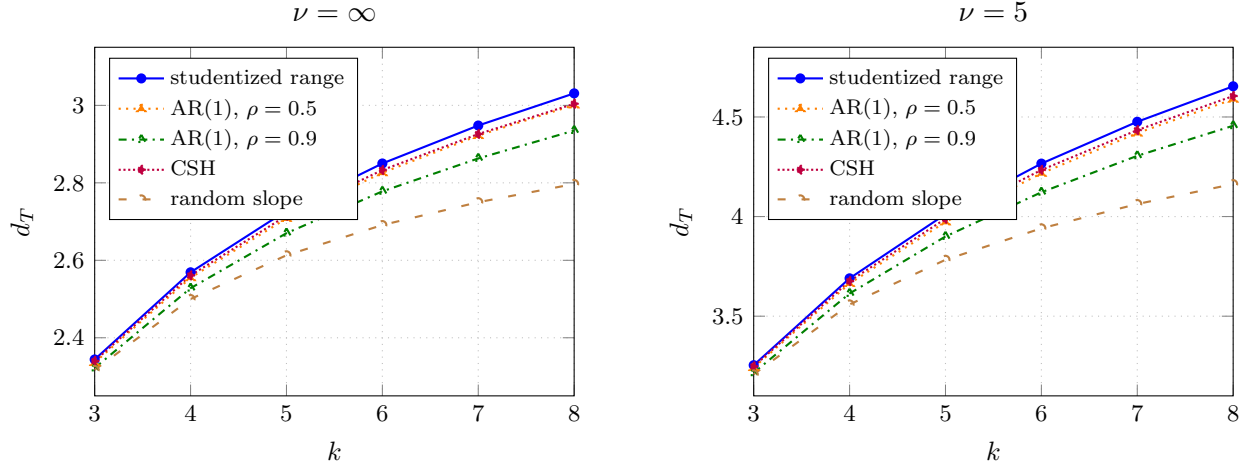
\begin{figure}[ht]
\centering
\begin{tikzpicture}
\begin{axis}[
  width=0.48\textwidth, height=6.2cm,
  title={\(\nu=\infty\)},
  xlabel={\(k\)}, ylabel={\(d_T\)},
  xmin=3, xmax=8, ymin=2.25, ymax=3.15, xtick={3,...,8},
  legend pos=north west, legend style={font=\scriptsize, cells={anchor=west}},
  grid=major, grid style={dotted},
  tick label style={font=\footnotesize}, label style={font=\small},
]
\addplot[thick, blue, mark=*, mark size=1.5pt] coordinates
  {(3,2.344) (4,2.569) (5,2.728) (6,2.850) (7,2.948) (8,3.031)};
\addlegendentry{studentized range}
\addplot[thick, orange, dotted, mark=triangle*, mark size=1.8pt] coordinates
  {(3,2.336) (4,2.555) (5,2.708) (6,2.826) (7,2.922) (8,3.001)};
\addlegendentry{AR(1), \(\rho=0.5\)}
\addplot[thick, green!50!black, dashdotted, mark=triangle, mark size=1.8pt] coordinates
  {(3,2.322) (4,2.528) (5,2.670) (6,2.778) (7,2.863) (8,2.935)};
\addlegendentry{AR(1), \(\rho=0.9\)}
\addplot[thick, purple, densely dotted, mark=diamond*, mark size=1.8pt] coordinates
  {(3,2.340) (4,2.561) (5,2.715) (6,2.833) (7,2.926) (8,3.004)};
\addlegendentry{CSH}
\addplot[thick, brown, loosely dashed, mark=o, mark size=1.5pt] coordinates
  {(3,2.317) (4,2.500) (5,2.613) (6,2.691) (7,2.750) (8,2.797)};
\addlegendentry{random slope}
\end{axis}
\end{tikzpicture}
\hfill
\begin{tikzpicture}
\begin{axis}[
  width=0.48\textwidth, height=6.2cm,
  title={\(\nu=5\)},
  xlabel={\(k\)}, ylabel={\(d_T\)},
  xmin=3, xmax=8, ymin=3.1, ymax=4.85, xtick={3,...,8},
  legend pos=north west, legend style={font=\scriptsize, cells={anchor=west}},
  grid=major, grid style={dotted},
  tick label style={font=\footnotesize}, label style={font=\small},
]
\addplot[thick, blue, mark=*, mark size=1.5pt] coordinates
  {(3,3.254) (4,3.690) (5,4.012) (6,4.266) (7,4.476) (8,4.654)};
\addlegendentry{studentized range}
\addplot[thick, orange, dotted, mark=triangle*, mark size=1.8pt] coordinates
  {(3,3.242) (4,3.664) (5,3.973) (6,4.217) (7,4.419) (8,4.589)};
\addlegendentry{AR(1), \(\rho=0.5\)}
\addplot[thick, green!50!black, dashdotted, mark=triangle, mark size=1.8pt] coordinates
  {(3,3.216) (4,3.614) (5,3.900) (6,4.122) (7,4.305) (8,4.456)};
\addlegendentry{AR(1), \(\rho=0.9\)}
\addplot[thick, purple, densely dotted, mark=diamond*, mark size=1.8pt] coordinates
  {(3,3.248) (4,3.676) (5,3.989) (6,4.234) (7,4.434) (8,4.605)};
\addlegendentry{CSH}
\addplot[thick, brown, loosely dashed, mark=o, mark size=1.5pt] coordinates
  {(3,3.208) (4,3.559) (5,3.783) (6,3.941) (7,4.063) (8,4.163)};
\addlegendentry{random slope}
\end{axis}
\end{tikzpicture}
\caption{Tukey critical values \(d_T\) against the number of treatments. 
 Left, \(\nu=\infty\); right, \(\nu=5\). }
\label{fig:tukey}
\end{figure}

\section{$N$-way linear mixed model}
\label{sec:nway}

An \(N\)-way layout has \(N\) fixed factors. Of these, \(M\) are within-subjects
factors: every subject is observed at every combination of their levels.
Subjects are nested in the cells of the remaining between-subjects factors
and crossed with the \(M\) within-subjects factors. The case \(M=1\) is
repeated measures on one factor, the case \(M=2\) is repeated measures on
two factors, and the formulae below are the same for a general \(M\).

Let \(C\) be the number of between-subject cells and \(n\) be  the number
of subjects in each cell. Denote the numbers of levels
of the within-subjects factors as \(b_1,\ldots,b_M\). The Dunnett factor is within-factor \(1\),
with control level \(0\) and \(q=b_1-1\) comparisons. A slice is one
between-subject cell together with one combination of levels of
within-factors \(2,\ldots,M\). There are
\begin{equation}
\label{eq:K-slices}
K
=
C\prod_{r=2}^{M} b_r
\end{equation}
slices. \(K=C\) if \(M=1\).

\subsection{Marginal comparisons}
\label{sec:nway-marg}

A marginal comparison on a between-subjects factor averages the other
factors with equal weights. Subjects in different between-cells are
independent, so the marginal means of two levels are independent whenever
those levels use different subjects. The Dunnett correlation is then \(1/2\)
for every subject variance, by Proposition~\ref{prop:common}. The denominator
is the between-subjects mean square, on \(\nu=C(n-1)\) degrees of freedom
when one between-subjects factor accounts for all \(C\) cells and the design
is balanced. Further within-subjects levels change the standard error. They
do not add denominator degrees of freedom.

A marginal comparison on a within-subjects factor averages the other factors
first. Every random effect that does not involve the compared factor is then
common to all of its levels and cancels. Under compound symmetry within
subjects the correlation returns to \(1/2\), and the denominator is the
interaction of that factor with subjects. Integrals~\eqref{eq:6}
and~\eqref{eq:7} are~\eqref{eq:mm-one} and~\eqref{eq:mm-two}.

\subsection{The block integral}
\label{sec:mm-twoway}

Suppose every random effect that is constant across the Dunnett factor
cancels in the contrast, and no subject-by-Dunnett-factor variance component
remains to tie one slice to another. Each slice is then an equicorrelated
Dunnett family of size \(q\) at correlation \(1/2\), the slices are
independent, and they share one error mean square \(s^2\). That mean square
is the interaction of all \(M\) within-subjects factors with subjects, within
between-cells. In the balanced Gaussian layout it is independent of the
contrasts and
\[
\nu
=
C(n-1)\prod_{r=1}^{M}(b_r-1),
\qquad
w=s/\sigma,
\]
with \(\sigma^2\) the residual variance and \(f_\nu\) the density
\eqref{eq:wdensity}. Conditional on \(w\), independence of the slices
multiplies \(K\) copies of the inner integral of~\eqref{eq:mm-two}:
\begin{equation}
\label{eq:mm-blocks}
P(d'')
=
\int_0^{\infty}
\left(
\int_{-\infty}^{\infty}
\bigl[
\Phi(v+d''w\sqrt{2})-\Phi(v-d''w\sqrt{2})
\bigr]^{q}
\phi(v)\,\d v
\right)^{K}
f_\nu(w)\,\d w.
\end{equation}
The one-sided integral~\eqref{eq:6} replaces the difference of \(\Phi\) by
\(\Phi(v+d'w\sqrt{2})\). The power \(K\) stands outside the integral over the
control residual, because the slices have different control residuals, and
inside the integral over \(w\), because the slices share \(s\).

The studentized-range family on the same slices replaces the inner integral
by the Pearson-Hartley integral at threshold \(d_T\sqrt{2}\):
\begin{equation}
\label{eq:tukey-blocks}
\Prob(Q\le d_T\sqrt{2})
=
\int_0^{\infty}
\left(
k\int_{-\infty}^{\infty}
\phi(x)\,
\bigl[\Phi(x+d_T w\sqrt{2})-\Phi(x)\bigr]^{k-1}
\,\d x
\right)^{K}
f_\nu(w)\,\d w.
\end{equation}
Here \(k=b_1\), and \(K\) and \(\nu\) are those of~\eqref{eq:mm-blocks}. The
power \(K\) is placed in the same way.

If a subject interaction with the Dunnett factor is present, slices that
share the same subjects are tied by that interaction. The \(C\)
between-subject cells stay independent, so the integral is a power \(C\) of
one cell's probability, inside the integral over the shared \(w\). Inside a
cell the probability is the conditioning integral of
Proposition~\ref{prop:cond}. For \(M=2\) that integral is~\eqref{eq:mm-cross}.
Setting the tying variances to zero returns~\eqref{eq:mm-blocks} and
\eqref{eq:tukey-blocks}.

Table~\ref{tab:blocks} solves~\eqref{eq:mm-blocks} at \(P=0.95\) and
\(\nu=\infty\), with \(K=a\) independent groups and \(q\) comparisons at
correlation \(1/2\) in each group. The root \(d''\) is the critical value for
that family: independence of the groups makes the coverage of all \(a\)
groups together the \(a\)-th power of the coverage inside one group, and
\(d''\) is chosen so that this power equals \(0.95\).

The value \(d_q\) is the same root with \(a=1\). It covers one group with
probability \(0.95\). Applied unchanged to all \(a\) groups, the group
coverages multiply, and the simultaneous family is covered with probability
\(0.95^{a}\).

The value \(d_{aq}\) is the root for a single equicorrelated family of
\(aq\) contrasts, every correlation equal to \(1/2\). In the block-diagonal
family the correlation is \(1/2\) inside a group and zero between groups, so
\(d_{aq}\) lies below \(d''\). The last column is the coverage of this
smaller constant on the block-diagonal family. It runs from \(0.940\) to
\(0.946\). In the first row, \(a=2\) and \(q=2\), the three constants are
\(d_q=2.212\), \(d_{aq}=2.442\) and \(d''=2.473\): the one-group constant
covers \(0.95^{2}=0.9025\) of the family, and the equicorrelated constant
covers \(0.946\).

\begin{table}[ht]
\centering
\caption{Two-sided \(d''\) at \(P=0.95\) and \(\nu=\infty\) for simultaneous
simple effects 
in \(a\) independent groups.}
\label{tab:blocks}
\begin{tabular}{rrrrrr}
\toprule
\(a\) & \(q\) & \(d''\) & \(d_q\) & \(d_{aq}\) & coverage of \(d_{aq}\) \\
\midrule
2 & 2 & 2.473 & 2.212 & 2.442 & 0.946 \\
2 & 3 & 2.602 & 2.349 & 2.567 & 0.945 \\
2 & 5 & 2.757 & 2.511 & 2.716 & 0.944 \\
3 & 3 & 2.741 & 2.349 & 2.686 & 0.941 \\
4 & 2 & 2.714 & 2.212 & 2.652 & 0.940 \\
\bottomrule
\end{tabular}
\end{table}

Table~\ref{tab:blocks-check} checks the block integrals~\eqref{eq:mm-blocks}
and~\eqref{eq:tukey-blocks} against the unreduced cube probability. The
roots were obtained from the double integrals by adaptive quadrature, the
inner integral on the real line and the outer integral on the chi-square
probability scale of \(w\). The cube probability was then computed at each
root by the Genz-Bretz algorithm~\cite{GenzBretz2009} in \texttt{mvtnorm} package version 1.3-3 (R version 4.5.1). For the Dunnett family the covariance is block diagonal
with \(K\) blocks of size \(q\) and off-diagonal entries \(1/2\), a matrix of
dimension \(Kq\). For the range family it is block diagonal with \(K\) blocks
of the \(k(k-1)/2\) pairwise contrasts of \(k\) independent means, a singular
matrix of rank \(K(k-1)\), and the bounds are \(\pm d_T\) times the standard
deviation of each contrast. The absolute error tolerance was \(10^{-6}\).
Every cube probability lies within \(1.6\times 10^{-5}\) of \(0.95\), and
the Dunnett rows lie within \(5\times 10^{-6}\). The largest deviation is the
\(20\)-dimensional range family \(K=2\), \(k=5\) at \(\nu=20\). The same
roots were recomputed with \texttt{NIntegrate} in Wolfram Mathematica 14.3, and the two
quadratures agree to \(10^{-6}\). At \(K=1\) the range block
integral~\eqref{eq:tukey-blocks} reproduces the studentized range
distribution function at the \(0.95\) point to seven decimals for
\(k=3,4,5\) and \(\nu=5,20,\infty\). The rows \(\nu=\infty\) of the Dunnett
part are the column \(d''\) of Table~\ref{tab:blocks}. For \(M=1\) the
table reads \(K=a\) groups; for \(M=2\) it reads
\(K=b\) levels of the second within-subjects factor, on
\(\nu=p(b-1)(n-1)\).

\begin{table}[ht]
\centering
\caption{Roots of the block integrals at \(P=0.95\) and the probability of
the block-diagonal cube at each root, computed by \texttt{mvtnorm}. Upper
part: Dunnett integral~\eqref{eq:mm-blocks}, \(K\) slices of \(q\)
comparisons. Lower part: studentized range integral~\eqref{eq:tukey-blocks},
\(K\) slices of \(k\) levels, root on the scale \(d_T=q/\sqrt2\).}
\label{tab:blocks-check}
\begin{tabular}{rrrrrrrr}
\toprule
& & \multicolumn{2}{c}{\(\nu=\infty\)} & \multicolumn{2}{c}{\(\nu=20\)} & \multicolumn{2}{c}{\(\nu=5\)} \\
\(K\) & \(q\) & \(d''\) & cube & \(d''\) & cube & \(d''\) & cube \\
\midrule
2 & 2 & 2.4728 & 0.95000 & 2.6947 & 0.95000 & 3.5619 & 0.95000 \\
2 & 3 & 2.6022 & 0.95000 & 2.8526 & 0.95000 & 3.8288 & 0.95000 \\
2 & 5 & 2.7569 & 0.95000 & 3.0435 & 0.95000 & 4.1532 & 0.95000 \\
3 & 3 & 2.7414 & 0.95000 & 3.0296 & 0.95000 & 4.1406 & 0.95000 \\
4 & 2 & 2.7137 & 0.95000 & 2.9979 & 0.95000 & 4.0934 & 0.95000 \\
\midrule
\(K\) & \(k\) & \(d_T\) & cube & \(d_T\) & cube & \(d_T\) & cube \\
\midrule
2 & 3 & 2.5992 & 0.95000 & 2.8456 & 0.95000 & 3.7999 & 0.95000 \\
2 & 4 & 2.8128 & 0.95000 & 3.1098 & 0.94999 & 4.2451 & 0.95000 \\
3 & 3 & 2.7392 & 0.95000 & 3.0239 & 0.95000 & 4.1173 & 0.95000 \\
2 & 5 & 2.9635 & 0.95001 & 3.3000 & 0.94998 & 4.5697 & 0.94999 \\
4 & 3 & 2.8347 & 0.95000 & 3.1479 & 0.95000 & 4.3399 & 0.95000 \\
\bottomrule
\end{tabular}
\end{table}


Table~\ref{tab:blocks-time} times one evaluation at the root for \(P=0.95\)
and \(\nu=20\). The block integral stays near a tenth of a second. The cube's
time and reported error grow with the dimension, and above dimension \(10\)
the run stops at its point limit.

\begin{table}[ht]
\centering
\caption{One evaluation of the block integral and of the unreduced cube at
the root for \(P=0.95\) and \(\nu=20\). Upper part: Dunnett
integral~\eqref{eq:mm-blocks}, dimension \(Kq\). Lower part: range
integral~\eqref{eq:tukey-blocks}, \(Kk(k-1)/2\) contrasts of rank
\(K(k-1)\). The column prob.\ is the Genz-Bretz probability at the root
and err its reported absolute error. The columns block and cube are the
times in seconds for the block integral and the unreduced cube, and the
factor is the ratio of those times.}
\label{tab:blocks-time}
\begin{tabular}{rrrrrrrrr}
\toprule
& & & & \multicolumn{2}{c}{unreduced cube} & \multicolumn{2}{c}{time (s)} & \\
\cmidrule(lr){5-6} \cmidrule(lr){7-8}
\(K\) & \(q\) & dim & \(d''\) & prob. & err & block & cube & factor \\
\midrule
2 & 2 & 4 & 2.6947 & 0.950001 & \(7\times10^{-7}\) & 0.11 & 0.78 & 7 \\
2 & 5 & 10 & 3.0435 & 0.949996 & \(1\times10^{-5}\) & 0.12 & 1.33 & 11 \\
4 & 5 & 20 & 3.3406 & 0.950005 & \(2\times10^{-5}\) & 0.11 & 2.35 & 21 \\
6 & 6 & 36 & 3.5740 & 0.949984 & \(6\times10^{-5}\) & 0.09 & 4.00 & 44 \\
8 & 8 & 64 & 3.7906 & 0.949967 & \(8\times10^{-5}\) & 0.09 & 6.89 & 77 \\
10 & 10 & 100 & 3.9551 & 0.950052 & \(1\times10^{-4}\) & 0.08 & 10.95 & 137 \\
\midrule
\(K\) & \(k\) & dim & \(d_T\) & prob. & err & block & cube & factor \\
\midrule
2 & 3 & 6 & 2.8456 & 0.950001 & \(2\times10^{-6}\) & 0.09 & 0.78 & 9 \\
2 & 5 & 20 & 3.3000 & 0.949984 & \(5\times10^{-5}\) & 0.09 & 1.16 & 13 \\
4 & 5 & 40 & 3.5947 & 0.950014 & \(4\times10^{-5}\) & 0.08 & 2.09 & 26 \\
3 & 8 & 84 & 3.8431 & 0.949917 & \(2\times10^{-4}\) & 0.08 & 3.18 & 40 \\
6 & 8 & 168 & 4.1241 & 0.949988 & \(2\times10^{-4}\) & 0.08 & 7.19 & 90 \\
\bottomrule
\end{tabular}
\end{table}

\section{Conclusions}
\label{sec:apply}

Dunnett~\cite{Dunnett1964} treated comparisons with a control in the balanced
model with compound symmetry: the correlation is \(1/2\) for every
intra-class correlation, the denominator is the treatment-by-subject
interaction, and at a fixed \(\nu\) the critical constant is the ordinary
analysis-of-variance constant. This paper determines which further designs
and covariances remain on that integral, and the direction in which the
constant moves when the contrast correlation leaves \(1/2\). The same
determination is made for Tukey's procedure. Under any spherical covariance
the pairwise correlation is the ordinary studentized-range correlation, and
at \(\nu=(k-1)(n-1)\) the critical constant is the ordinary Tukey constant.

The structures in Section~\ref{sec:cov-1955} and Table~\ref{tab:tukey} change
the constant. AR(1), the Toeplitz band, alternating ante-dependence, and the
random slope have roots below the ordinary value, so Table~2a and the root
of~\eqref{eq:7} are conservative for them. Heterogeneous compound symmetry
and heteroscedastic groups have roots above that value. Table~2a is then
anticonservative at \(p=2\), and the root of~\eqref{eq:7} is
anticonservative at every \(p\). For Tukey's procedure every computed
non-spherical covariance lowers the root, so the ordinary studentized-range
constant is conservative. In the four-time autoregressive example of
Section~\ref{sec:cov-ar}, a least-squares product approximation moves
the root by \(0.003\), and correlation \(1/2\) moves it by \(0.029\).

The reduction is exact when the contrast correlation does not depend on an
unknown variance-component ratio and the studentizing mean square is one
scaled chi-square independent of the contrasts. The integral is then
\eqref{eq:mvt-lambda} for a shared control term, which covers unequal
replication in which the blocks cancel, independent groups, and
heteroscedastic groups with known variance ratios, and it is
\eqref{eq:tukey-student} for the range under sphericity. The product form
is not supported for incomplete subject sets once the overlap is asymmetric and there
are four or more treatments, and for autoregressive or unstructured
baseline comparisons with four or more follow-up times. The range integral
stops when a pairwise variance depends on the pair, or when a disjoint pair
acquires a nonzero correlation. Separate variance estimates in each group
remove the single scale. The probability is then the unreduced
integral~\eqref{eq:mvt}, or the multivariate-\(t\) cube of
Section~\ref{sec:tukey-cov}. 

\bibliographystyle{plain}
\bibliography{dunnett_integral7_cs}   

\end{document}